\documentclass{amsart}
\usepackage{amsthm,amsmath,amsfonts,amssymb}
\usepackage[all]{xy}
\usepackage{graphicx}
\usepackage{color}
\usepackage[pagebackref, colorlinks=true, linkcolor=blue, citecolor=blue]{hyperref}
\providecommand{\eprint}[2][]{\href{http://arxiv.org/abs/#2}{arXiv:#2}}

\newtheorem{theorem}{Theorem}[section]
\newtheorem{lemma}[theorem]{Lemma}
\newtheorem{proposition}[theorem]{Proposition}
\newtheorem{corollary}[theorem]{Corollary}

\theoremstyle{definition}
\newtheorem{definition}[theorem]{Definition}
\newtheorem{example}[theorem]{Example}
\newtheorem{notation}[theorem]{Notation}

\theoremstyle{remark}
\newtheorem{remark}[theorem]{Remark}

\numberwithin{equation}{section}

\newcommand{\K}{\mathbb{K}}% blackboard math
\newcommand{\C}{\mathbb{C}}% blackboard math
\newcommand{\Z}{\mathbb{Z}}% blackboard math

\newcommand{\sM}{\mathcal{M}}

\newcommand{\Id}{\operatorname{Id}}

\newcommand{\de}{\partial}
\newcommand{\debar}{\overline{\partial}}

\newcommand{\bi}{\boldsymbol{i}}
\newcommand{\bl}{\boldsymbol{l}}

\newcommand{\Def}{\operatorname{Def}}

\newcommand{\MC}{\operatorname{MC}}
\newcommand{\Hom}{\operatorname{Hom}}
\newcommand{\End}{\operatorname{End}}
\newcommand{\Aff}{\operatorname{Aff}}
\newcommand{\Coker}{\operatorname{coker}}
\newcommand{\Aut}{\operatorname{Aut}}
\newcommand{\Grass}{\operatorname{Grass}}

\newcommand{\Hilb}{\operatorname{Hilb}}

\newcommand{\contr}{{\mspace{1mu}\lrcorner\mspace{1.5mu}}}
\newcommand{\cone}{\operatorname{cone}}

\newcommand{\Art}{\mathbf{Art}}
\newcommand{\Set}{\mathbf{Set}}

\newdir^{ (}{{}*!/-5pt/@^{(}}
\newdir^{ )}{{}*!/-5pt/@_{(}}

\begin{document}

\title{Formal Abel-Jacobi maps}
\author{Domenico Fiorenza}
\address{\newline
Universit\`a degli studi di Roma ``La Sapienza'',\hfill\newline
Dipartimento di Matematica \lq\lq Guido
Castelnuovo\rq\rq,\hfill\newline
P.le Aldo Moro 5,
I-00185 Roma, Italy.}
\email{fiorenza@mat.uniroma1.it}
\urladdr{www.mat.uniroma1.it/people/fiorenza/}
\author{Marco Manetti}
\email{manetti@mat.uniroma1.it}
\urladdr{www.mat.uniroma1.it/people/manetti/}

\subjclass[2010]{18G55, 13D10}
\keywords{Homotopical algebra, differential graded Lie algebras}

\date{October 28, 2016}
\maketitle

\begin{abstract}
We realize the infinitesimal Abel-Jacobi map as a morphism of formal deformation theories, realized as a morphism in the homotopy category of differential graded Lie algebras. 
The whole construction is carried out in a general setting, of which the classical Abel-Jacobi map is a special example.
\end{abstract}

\section*{Introduction}

One of the main themes in deformation theory in the last 20 years is the proof that certain classical 
maps in complex algebraic geometry
may be interpreted as the ``classical'' restrictions of some morphisms of deformation theories.
For the Griffiths period map and the Abel-Jacobi map, which are the morphism we are concerned with in the present paper, this has been investigated in 
\cite{buchFlenner,FioMart,semireg2011,ManettiSeattle,Pri}.

The notion of deformation theory, intended as above, has been only recently 
clarified at different levels of generality and abstraction by means of ideas and methods of homotopical algebra and derived algebraic geometry, cf. \cite{Lur,Pri10}. 
In this paper, by a deformation theory over a field of characteristic $0$ we shall intend the 
homotopy type of a DG-Lie algebra, following the well-known principle by Deligne and Drinfeld that   
in characteristic 0, a deformation problem is controlled by a differential graded Lie algebra, with quasi-isomorphic DG-Lie algebras giving the same deformation theory. A morphism of deformation theories is a morphism in the homotopy category of DG-Lie algebras, i.e., in the category obtained by formally inverting all the quasi-isomorphisms, see e.g. \cite[Def. 1.2.1]{Hov99}.

The relation between DG-Lie algebras $L$ and local moduli spaces is obtained 
via Maurer-Cartan equation
and gauge equivalence, cf. \cite{GoMil1}.  
The first cohomology group $H^1(L)$ 
is equal to the Zariski tangent space
of the local moduli space, while $H^2(L)$ is a complete obstruction space.
A morphism of deformation theories, i.e.,  a morphism in the homotopy  category
of DG-Lie algebras induces a well defined morphism  in
cohomology, giving, in degrees 1 and 2, the tangent and obstruction map, respectively.

Particularly interesting is the case of unobstructed deformation theories with nontrivial obstruction spaces: in fact,  every morphism from a deformation theory $L$ into an unobstructed deformation 
theory $M$ provides an obstruction map annihilating every obstruction of $L$. 
We refer to \cite{ManettiSeattle} for more details  and examples. Particularly relevant for this paper is the example of the Grassmann formal moduli space of subcomplexes up to homotopy of a fixed complex, see 
Section~\ref{sec.grassmann}. If $V$ is a complex of vector spaces and $F\subseteq V$ is a subcomplex such that the natural map $H^*(F)\to H^*(V)$ is injective, then the the corresponding ``classical'' moduli space is unobstructed although in general it has a nontrivial obstruction space, cf. \cite{Carmelo1}.  

While the Griffiths' period map has been treated extensively in \cite{Yuka,Carmelo2,FMperiods,FioMart}, here we focus our attention to 
the local Abel-Jacobi maps. 
More precisely we want to give a precise algebraic description of the following heuristic picture. 
Let $X$ be a compact K\"{a}hler manifold and let $Z\subset X$ be a smooth compact closed submanifold of codimension $p$. Denoting by $\Def_X$ the deformations of $X$ and by $\Def_{(X,Z)}$ the deformations of the pair $(X,Z)$, it has been proved in \cite{ManettiSemireg} that $\Hilb_{Z|X}$, the deformation theory of embedded deformations of $Z$ inside $X$, is the homotopy fiber of the natural forgetful morphism 
$\alpha\colon\Def_{(X,Z)}\to\Def_{X}$.  For every (small) deformation $X_\xi$ of $X$, the Gauss-Manin connection provides a natural isomorphism $H^*(X_{\xi},\C)\simeq H^*(X,\C)$ and the $p$th period map is defined as 
\[ \mathfrak{p}\colon \Def_X\to \Grass(H^*(X,\C)),\qquad X_{\xi}\mapsto F^pH^*(X_{\xi},\C)=\oplus_{i\ge p}H^{i,*}(X_{\xi})\,.\] 
Following \cite{Yuka,FMperiods}, the period map $\mathfrak{p}$ 
depends explicitly and functorially on the contraction product of differential forms, or  currents, with vector fields of type $(1,0)$. 
Denoting by $\pi\colon Q\to \Grass(H^*(X,\C))$ the tautological quotient bundle, with a  shift in the degree of the fibers of magnitude $2p-1$, the submanifold $Z$ gives a lifting of the period map
\[ \mathfrak{q}\colon \Def_X\to Q,\qquad X_{\xi}\mapsto \left(F^pH^*(X_{\xi},\C),
[Z]\in \frac{H^*(X,\C)}{F^pH^*(X_{\xi},\C)}[2p-1]\right)
\,,\] 
such that, if  the cohomology class of $Z$ remains of type $(p,p)$ in $X_{\xi}$, then $\mathfrak{q}(X_{\xi})$ is contained in the image of the zero section $0\colon \Grass(H^*(X,\C))\to Q$. A description of the lift $\mathfrak{q}$ avoiding the noncanonical choice of a cocycle representing the cohomology class of $Z$ and independent of the particular K\"ahler metric on $X$ can be obtained by noticing that the Hodge filtration on de Rham cohomology can be equally obtained by considering the de Rham complex $D_X^{*,*}$ of currents instead of differential forms. The integration map $\int_Z$ is then a well defined closed element in $D_X^{*,*}$. The degree shifting  
is geometrically motivated by the fact that Abel-Jacobi maps are defined in terms of integrals of differential forms on 
submanifolds of real codimension $2p-1$. 

In particular we have a commutative diagram
\begin{equation*}
\xymatrix{\Def_{(X,Z)}\ar[d]^{\alpha}\ar[r]^-{\mathfrak{p}}&\Grass(H^*(X,\C))\ar[d]^0\\
\Def_X\ar[r]^{\mathfrak{q}}&Q}\end{equation*}
and therefore a morphism from $\Hilb_{Z|X}$ to the homotopy fiber of the $0$ section. By a general argument in homotopical algebra the homotopy fiber of a section  is the loop space of the fiber, and then  we have  a map 
\[ \mathcal{AJ}\colon \Hilb_{Z|X}\to \Omega \left(\frac{H^*(X,\C)}{F^pH^*(X,\C)}[2p-1]\right)\simeq\frac{H^*(X,\C)}{F^pH^*(X,\C)}[2p-2].\]
which at the level of tangent and obstruction spaces gives the maps:
\[ AJ_1\colon H^0(Z,N_{Z|X})\to H^1\left(\frac{H^*(X,\C)}{F^pH^*(X,\C)}[2p-2]\right)=
\bigoplus_{i<p}H^{i,2p-1-i}(X),\]
\[ AJ_2\colon H^1(Z,N_{Z|X})\to H^2\left(\frac{H^*(X,\C)}{F^pH^*(X,\C)}[2p-2]\right)=
\bigoplus_{i<p}H^{i,2p-i}(X)\,.\]
One easily recognizes the codomain of $AJ_1$ as the tangent space of the $p$th intermediate Jacobian of $X$, cf. \cite{Green},  and the codomain of $AJ_2$ as the codomain of the semiregularity map, cf. \cite{semireg2011}.  A closer inspection reveals that the maps $AJ_1$ and $AJ_2$ are indeed the usual infinitesimal Abel-Jacobi and  semiregularity maps, respectively. A pleasant consequence of the unobstructedness of the Grassmannian is then that the obstruction space of $\Hilb_{Z|X}$ is contained in the kernel of $AJ_2$, and one gets a proof of the classical principle ``semiregularity kills obstructions'' \cite{bloch,semireg2011}.

\medskip 

The main goal of this paper is to give a rigorous proof of the above results, while providing an abstract algebraic analogue of the above geometric situation which is adaptable to more general deformation problems. Up to a suitable shifting degree,  the role of 
$D_X^{*,*}$ is played by an arbitrary complex of vector spaces $V$, the $p$th stage of the Hodge filtration $F^pD_X^{*,*}$  is replaced by an arbitrary subcomplex $F\subseteq V$, the role of $\int_Z$ is played by a  cocycle  
$v\in F$ and $\Def_X$ is replaced by an arbitrary DG-Lie algebra $\mathfrak{g}$. In the abstract setting, the contraction product is replaced by 
a linear map $\bi\colon \mathfrak{g}\to \Hom^*_{\K}(V,V)$ of degree $-1$, called \emph{Cartan homotopy}, such that for every $x,y\in \mathfrak{g}$ 
\[ [\bi_x,\bi_y]=0,\qquad \bi_{[x,y]}=[\bi_x,d\bi_y],\qquad
(d\bi_x+\bi_{dx})(F)\subseteq F\,.\]  
Finally the morphism $\alpha\colon \Def_{(X,Z)}\to \Def_{X}$ is replaced by a morphism 
$\alpha\colon \tilde{\mathfrak{g}}\to \mathfrak{g}$ of DG-Lie algebras such that 
$\bi_{\alpha(x)}v=0$ for every $x\in \tilde{\mathfrak{g}}$. For simplicity of exposition we shall only consider here the case where $\alpha$ is injective; the general case is a straightforward generalization.

\tableofcontents

\bigskip
\section{Notational setup}

We shall work over a fixed field $\K$ of characteristic 0. Unless otherwise specified the symbol $\otimes$ means the tensor product over $\K$. Every differential graded (DG) vector space $V$ over $\K$ is intended to have the differential $d\colon V\to V$ of degree $+1$.

Let $C$ be a finite category. 
We denote by $\Omega^{\bullet}(C,\K)$ the  differential graded commutative algebra of polynomial differential forms on the nerve of $C$, see e.g. \cite[Def. 2.1]{BG76}. More generally, for every differential graded graded vector space $V$ we set $\Omega^{\bullet}(C,V)=\Omega^{\bullet}(C,\K)\otimes V$.

A $C$-shaped diagram in a category $\sM$ is a covariant functor
$C\to \sM$. The collection $Fun(C,\sM)$ of all $C$-shaped diagrams in $\sM$ is naturally a category, with natural transformations as morphisms. When $C$ is a Reedy category and $\sM$ is a model category, then $Fun(C,\sM)$ is naturally equipped with a model category structure as well: the Reedy model structure. In this article we will not need the full generality of Reedy categories, but we will only deal with 
fibrations and weak equivalences, and with shapes $C$ that are finite posets, which  may be considered as Reedy categories in two natural ways: the direct way, where every non-identity arrow raises degree, and the inverse way, where every non-identity arrow lowers degree. 
Therefore, for the convenience of the reader, we recall here the structure of fibrations and weak equivalences in the  direct and the inverse Reedy model category structure on 
$Fun(C,\sM)$ for $C$ a finite poset, addressing the interested reader to \cite[Ch. 12]{Hir03} for the general definition and basic properties of Reedy model categories.

\begin{remark}\label{rem.reedy}
Let $C$ be a finite poset and let $\sM$ be a model category. 
In the direct Reedy model category structure on  $Fun(C,\sM)$, a morphism of diagrams 
\[ f\colon X\to Y,\qquad X,Y\in Fun(C,\sM),\] 
is: 
\begin{itemize}
\item a \emph{weak equivalence}  if $f_i\colon X_i\to Y_i$ is a weak equivalence for every $i\in C$;
\item a \emph{fibration} if $f_i\colon X_i\to Y_i$ is a fibration for every $i\in C$.
\end{itemize}
In the inverse Reedy model category structure on  $Fun(C,\sM)$, a morphism of diagrams 
\[ f\colon X\to Y,\qquad X,Y\in Fun(C,\sM),\] 
is: 
\begin{itemize}
\item a \emph{weak equivalence}  if $f_i\colon X_i\to Y_i$ is a weak equivalence for every $i\in C$;
\item a \emph{fibration} if 
\[
X_i\to \lim_{j>i}X_j\times_{\lim_{j>i}Y_j} Y_i
\]
is a fibration for every $i\in C$.
\end{itemize}
In particular both model structures have the same homotopy category. Recall that 
two diagrams $X,Y\in Fun(C,\sM)$ are  homotopy equivalent if they are joined by a zigzag of weak equivalences. 
\end{remark}

In this paper we are mainly interested to the following cases:

\begin{enumerate} 

\item $C=\Delta^1=\{0\to 1\}$, i.e.,   every diagram is a  morphism. Thus  
\[ \Omega^{\bullet}(\Delta^1,\K)=\frac{\K[t_0,t_1,dt_0,dt_1]}{(t_0+t_1-1,dt_0+dt_1)}\,,\] 
the integration map has degree $-1$ and may by defined by the explicit formula
\[  \int_{\Delta^1}\colon \Omega^{\bullet}(\Delta^1,\K)\to \K,\qquad 
\int_{\Delta^1}t_0^at_1^bdt_1=\int_0^1(1-t)^at^bdt=\frac{a!\, b!}{(a+b+1)!}\,.\]
\item $C=\Delta^1\times \Delta^1=\left\{\begin{matrix}0,0&\to&1,0\\
\downarrow&&\downarrow\\
0,1&\to&1,1\end{matrix}\right\}$, i.e.,  every diagram is a commutative square, and 
\[ \Omega^{\bullet}(\Delta^1\times \Delta^1,\K)=\Omega^{\bullet}(\Delta^1,\K)\otimes \Omega^{\bullet}(\Delta^1,\K)\,.\]  

\item $\sM=\mathbf{DG}$ is the model category of differential graded vector spaces.

\item $\sM=\mathbf{DGLA}$ is the model category of differential graded Lie algebras.
\end{enumerate}

Weak equivalences and fibrations in the categories $\mathbf{DG},\mathbf{DGLA}$ are quasi-isomorphisms and surjective maps, respectively. Moreover, in the category $\mathbf{DG}$ the cofibrations are the injective maps. Therefore, unwinding 
Remark~\ref{rem.reedy} we see for instance that, in the inverse Reedy model structure,  a fibrant object $\mathfrak{g}_\bullet$ in $Fun(\Delta^1,\mathbf{DGLA})$ is a surjective DG-Lie algebra morphism $\mathfrak{g}_0\to \mathfrak{g}_1$, and a  
fibrant object $\mathfrak{g}_{\bullet,\bullet}$ in $Fun(\Delta^1\times \Delta^1,\mathbf{DGLA})$ is a commutative square
\[ \xymatrix{\mathfrak{g}_{0,0}\ar[d]\ar[r]&\mathfrak{g}_{1,0}\ar[d]\\
\mathfrak{g}_{0,1}\ar[r]&\mathfrak{g}_{1,1}}\]
such that $\mathfrak{g}_{0,1}\to\mathfrak{g}_{1,1}$, $\mathfrak{g}_{1,0}\to\mathfrak{g}_{1,1}$, and $\mathfrak{g}_{0,0}\to\mathfrak{g}_{0,1}\times_{\mathfrak{g}_{1,1}}\mathfrak{g}_{1,0}$ are surjective DG-Lie algebra morphisms.

\begin{notation}
As a matter of notation, we will write $\mathfrak{g}\xrightarrow{\sim}\mathfrak{h}$ and $\mathfrak{g}\twoheadrightarrow \mathfrak{h}$ to denote a weak equivalence and a fibration in $\mathbf{DGLA}$, respectively. The same notation will be used also for weak equivalences and fibrations in $\mathbf{DG}$.
\end{notation}

\bigskip
\section{Homotopy fibers of DG-Lie algebras}

It is well known, and easy to prove, that the model category  of differential graded Lie algebras is right proper, i.e., the pullback of a quasi-isomorphism along a fibration is again a quasi-isomorphism. Equivalently, the following holds, cf. \cite[13.1.2 and 13.3.4]{Hir03}:
\begin{lemma}[The coglueing theorem]\label{lem.coglueing}
For any commutative diagram of differential graded Lie algebras of the form
\[ \xymatrix{\mathfrak{g}\ar@{->>}[r]\ar[d]%^f
_{\wr}&\mathfrak{h}\ar[d]%^g
_{\wr}&\mathfrak{l}\ar[l]\ar[d]%^h
_{\wr}\\
\tilde{\mathfrak{g}}\ar@{->>}[r]&\tilde{\mathfrak{h}}&\tilde{\mathfrak{l}}\ar[l]}\]
the induced map 
\[ \mathfrak{g}\times_{\mathfrak{h}}\mathfrak{l}\to \tilde{\mathfrak{g}}\times_{\tilde{\mathfrak{h}}}\tilde{\mathfrak{l}}\]
is a weak equivalence.
\end{lemma}

Right properness allows us to define the homotopy fiber of a morphism $\mathfrak{g}_0\to \mathfrak{g}_1$ of DG-Lie algebras in the following easy way: take any commutative diagram of the form
\[ \xymatrix{\mathfrak{g}_0\ar[d]_{\wr}\ar[r]&\mathfrak{g}_1\ar[d]_{\wr}\\
\tilde{\mathfrak{g}}_0\ar@{->>}[r]&\tilde{\mathfrak{g}}_1}\]
extending $\mathfrak{g}_0\to \mathfrak{g}_1$, then define 
\[ hofib(\mathfrak{g}_\bullet)=\ker(\tilde{\mathfrak{g}}_0\twoheadrightarrow \tilde{\mathfrak{g}}_1)=\tilde{\mathfrak{g}}_0\times_{\tilde{\mathfrak{g}}_1} 0\;,\]
as an element of the homotopy category: by the coglueing theorem the quasi-isomorphism class of
$hofib(\mathfrak{g}_\bullet)$ is well defined. 

The underlying complex of $hofib(\mathfrak{g}_\bullet)$ is also a representative for the homotopy fiber of $\mathfrak{g}_\bullet$ in the model category of DG-vector spaces and so it is quasi-isomorphic, as a cochain complex, to the mapping cocone of $\mathfrak{g}_\bullet$ as a morphism of DG-vector spaces. In particular 
we have a long exact cohomology sequence
\[ \cdots \to H^i(hofib(\mathfrak{g}_\bullet))\to H^i(\mathfrak{g}_0)\xrightarrow{}H^i(\mathfrak{g}_1)\to H^{i+1}(hofib(\mathfrak{g}_\bullet))\to\cdots\;.\]

The above model for the homotopy fiber of a morphism of DG-Lie algebras has the unpleasant feature of depending on a noncanonical choice of the fibration $\tilde{\mathfrak{g}}_0\twoheadrightarrow \tilde{\mathfrak{g}}_1$. As a consequence of this, the above construction of $hofib(\mathfrak{g}_\bullet)$ is not functorial. However, we can remedy this, by considering the following  choice for $\tilde{\mathfrak{g}}_0\twoheadrightarrow \tilde{\mathfrak{g}}_1$, leading to a functorial representative for the homotopy fiber, which we will call 
the \emph{Thom-Whitney fiber}. Every morphism $\mathfrak{g}_\bullet=(\mathfrak{g}_0\xrightarrow{f} \mathfrak{g}_1)$ 
of differential graded Lie algebras extends to a commutative diagram
\begin{equation}\label{equ.canonicalfactorization}
\xymatrix{\mathfrak{g}_0\ar@(u,u)[rr]^f\ar@{^{ (}->}[r]_-i\ar@(d,l)[dr]^{\Id_{\mathfrak{g}_0}}&\mathfrak{g}_0\times ^{\vert_1}_{\mathfrak{g}_1} \Omega^\bullet(\Delta^1,\mathfrak{g}_1)\ar@{->>}[r]_-{\vert_0}\ar[d]^{\pi}&\,\mathfrak{g}_1\\
&\,\mathfrak{g}_0&}
\end{equation}
where $\vert_0,\vert_1\colon \Omega^\bullet(\Delta^1,\mathfrak{g}_1)\to \mathfrak{g}_1$ are the pullback maps along the two face maps $\Delta^0\to \Delta^1$, $\pi$ is the projection on the first factor and  $i=(\mathrm{Id}_{\mathfrak{g}_0},f)$. Notice that  $i$ is the right inverse of the surjective quasi-isomorphism $\pi$ and so $i$ is an injective 
quasi-isomorphism, while $\vert_0$ is surjective. This means that the kernel of $\vert_0$ is a  model for  $hofib(\mathfrak{g}_\bullet)$.
\begin{definition} For every morphism $\mathfrak{g}_\bullet=\{\mathfrak{g}_0\to \mathfrak{g}_1\}$ of DG-Lie algebras,  we denote 
\begin{align*}
 TW(\mathfrak{g}_\bullet)&:=\ker \left(\vert_0\colon \mathfrak{g}_0\times ^{\vert_1}_{\mathfrak{g}_1} \Omega^\bullet(\Delta^1,\mathfrak{g}_1)\to \mathfrak{g}_1\right)\\
 &= \{(\omega_0,\omega_1)\in \mathfrak{g}_0\times \Omega^\bullet(\Delta^1,\mathfrak{g}_1)\mid \omega_1\vert_0=0,\;  \omega_1\vert_1=f(\omega_0)\}
 \end{align*}
and call it  the \emph{Thom-Whitney model} for the homotopy fiber of $\mathfrak{g}_\bullet$.
\end{definition}

The above diagram \eqref{equ.canonicalfactorization} shows in particular  that for every  quasi-isomorphism $\mathfrak{g}_0\xrightarrow{\sim} \mathfrak{g}_1$ of DG-Lie algebras, there exists a DG-Lie algebra $\mathfrak{h}$ and a span of acyclic fibrations $\mathfrak{g}_0\stackrel{\sim}{\twoheadleftarrow} \mathfrak{h}\stackrel{\sim}{\twoheadrightarrow} \mathfrak{g}_1$.
By right properness, the same is true as soon as $\mathfrak{g}_0,\mathfrak{g}_1$ are homotopy equivalent differential graded Lie algebras.

\begin{remark} Via the isomorphism $\K[t,dt]\to \Omega^\bullet(\Delta^1,\K)$, $t\mapsto t_1$, the 
Thom-Whitney homotopy fiber of a morphism $f\colon \mathfrak{g}_0\to \mathfrak{g}_1$ becomes
\[TW(\mathfrak{g}_\bullet)=\{(x,y(t))\in \mathfrak{g}_0\times\mathfrak{g}_1[t,dt]\mid y(0)=0,\; y(1)=f(x)\}\,.\]
Notice that for a terminal morphism $\mathfrak{g}\to 0$ one has $TW(\mathfrak{g}\to 0)=\mathfrak{g}$.
\end{remark}

\begin{lemma}\label{lemma.tw-quotient}
Assume that $\mathfrak{g}_\bullet=\{\mathfrak{g}_0\hookrightarrow \mathfrak{g}_1\}$ is the inclusion of a DG-Lie subalgebra  $\mathfrak{g}_0$ into a DG-Lie algebra 
$\mathfrak{g}_1$. Then we have a natural quasi-isomorphism of cochain complexes $TW(\mathfrak{g}_\bullet)\to (\mathfrak{g}_1/\mathfrak{g}_0)[-1]$ given by 
\[
(\omega_0,\omega_1)\mapsto \int_{\Delta^1} \omega_1 \mod \mathfrak{g}_0, 
\]
which, if $\mathfrak{g}_1$ has trivial bracket, is also a quasi-isomorphism of differential graded Lie algebras.
If there exists a morphism of DG-Lie algebras $\mathfrak{g}_1\to \mathfrak{g}_0$ which is a left inverse to the inclusion of $\mathfrak{g}_0$ in $\mathfrak{g}_1$, then 
the map
\[ \ker(\mathfrak{g}_1\to \mathfrak{g}_0)[-1]\to TW(\mathfrak{g}_\bullet),\qquad x_1\mapsto  dt\, x_1,\]
is a quasi-isomorphism of DG-Lie algebras, where $\ker(\mathfrak{g}_1\to \mathfrak{g}_0)[-1]$ is equipped with 
the trivial bracket.
\end{lemma}

\begin{proof} Straightforward, see e.g. \cite[Sec. 3]{FMcone}.
\end{proof}

\begin{remark}\label{rem.representative}
For later use, we point out that, if $\mathfrak{g}_0\hookrightarrow \mathfrak{g}_1$ is the inclusion of a DG-Lie subalgebra, the inverse of the isomorphism $H^*(TW(\mathfrak{g}_\bullet))\to H^{*-1}(\mathfrak{g}_1/\mathfrak{g}_0)$ induced by the quasi-isomorphism of complexes described in  Lemma \ref{lemma.tw-quotient} is given by
\begin{align*}
H^{*-1}(\mathfrak{g}_1/\mathfrak{g}_0)&\to H^*(TW(\mathfrak{g}_\bullet))\\
[x_1]&\mapsto [(d \tilde{x_1},d(t\,\tilde{x}_1))],
\end{align*}
where $\tilde{x}_1$ is any representative of $[x_1]$ in $\mathfrak{g}_1$.
\end{remark}

\begin{corollary}\label{loop} 
Let $\mathfrak{g}$ be a  differential graded Lie algebra and let 
$\Omega \mathfrak{g}:=TW(0\hookrightarrow \mathfrak{g})$ be the based loop space of $\mathfrak{g}$. 
Then  $\Omega \mathfrak{g}$ is homotopy abelian and quasi-isomorphic to the cochain complex $\mathfrak{g}[-1]$ endowed with the trivial bracket.
\end{corollary}
\begin{proof} Immediate from Lemma~\ref{lemma.tw-quotient}, since $\mathfrak{g}\to 0$ is a left inverse to $0\hookrightarrow\mathfrak{g}$.
.\end{proof}

\begin{lemma}\label{prop.iniettivoincohomologiaversohomotopyabelian} 
Let $\mathfrak{g}_\bullet=\{\mathfrak{g}_0\xrightarrow{f}\mathfrak{g}_1\}$ be a morphism of differential graded Lie algebras:
\begin{enumerate}

\item If $f$ is injective in cohomology and $\mathfrak{g}_1$  is quasi-isomorphic 
to an abelian DG-Lie algebra, then also $\mathfrak{g}_0$  is  quasi-isomorphic 
to an abelian DG-Lie algebra.

\item If $f$ is surjective in cohomology and $\mathfrak{g}_0$  is quasi-isomorphic 
to an abelian DG-Lie algebra, then also  $\mathfrak{g}_1$ is 
 quasi-isomorphic 
to an abelian DG-Lie algebra.
\end{enumerate}
\end{lemma}

\begin{proof} This is well known; a proof using homotopy classification of $L_{\infty}$-algebras can be found in \cite[Prop. 4.1]{KKP}, cf. also \cite[Lemma 1.10]{algebraicBTT}. It is however worth to give here 
a more elementary proof based only on the theory of differential graded Lie algebras.

(1) Since $\mathfrak{g}_1$ is homotopy equivalent to an abelian differential graded Lie algebra there exists a diagram of DG-Lie algebras of the form
\[ \xymatrix{&\mathfrak{l}\ar@{->>}[d]^g_{\wr}\ar@{->>}[r]^{\sim}_{\gamma}&\mathfrak{h}\\
\mathfrak{g}_0\ar[r]^f&\mathfrak{g}_1&}\]
with  
$\mathfrak{h}$ abelian with trivial differential. 
Considering the fiber product of $f$ and $g$ we get a new diagram 
\[ \xymatrix{\mathfrak{g}_0\times_{\mathfrak{g}_1} \mathfrak{l}\ar[r]^{\quad\beta}\ar@{->>}[d]^\alpha_{\wr}&\mathfrak{l}\ar@{->>}[r]^{\sim}_{\gamma}\ar@{->>}[d]^g_{\wr}&\mathfrak{h}\\
\mathfrak{g}_0\ar[r]^f&\mathfrak{g}_1&&}.\]
Since $f$ is injective in cohomology, also
$\beta$ injective in cohomology. 
It is now sufficient to consider  a suitable projection $\mathfrak{h}\xrightarrow{\delta} V$ of graded vector spaces 
such that the composition $\delta\gamma\beta$ is a quasi-isomorphism.\par

(2) Since $\mathfrak{g}_0$ is quasi-isomorphic to an abelian differential graded Lie algebra there exists a 
diagram of DG-Lie algebras of the form
\[ \xymatrix{\mathfrak{l}\ar@{->>}[d]_{\wr}^g\ar@{->>}[r]^{\sim}_{\gamma}&\mathfrak{g}_0\ar[r]^f&\mathfrak{g}_1\\
\mathfrak{h}&&}\]
where  
$\mathfrak{h}$ is an abelian DG-Lie algebra with trivial differential. 
Let's choose any morphism   $i\colon V\to \mathfrak{h}$  of graded vector spaces such that the 
composition 
\[ \xymatrix{H^*(V)= V\ar[r]^-{i}&\mathfrak{h}= H^*(\mathfrak{h})\ar[r]^-{\sim}_-{H^*(g)}&H^*(\mathfrak{l})\ar[r]^-{\sim}_-{H^*(\gamma)}&H^*(\mathfrak{g}_0)\ar@{->>}[r]^-{H^*(f)}&H^*(\mathfrak{g}_1)}\]
is an isomorphism. Taking the fiber product of $g$ and $i$  we get  a diagram 
\[ \xymatrix{V\times_{\mathfrak{h}}\mathfrak{l}\ar[r]^{\bar{i}}\ar[d]^{\bar{g}}_{\wr}&\mathfrak{l}\ar@{->>}[d]^g_{\wr}\ar@{->>}[r]^{\sim}_{\gamma}&\mathfrak{g}_0\ar[r]^f&\mathfrak{g}_1\\
V\ar[r]^i&\mathfrak{h}&&}\]
with  
$f\gamma \bar{i}$ a quasi-isomorphism of differential graded Lie algebras.
\end{proof}

\begin{corollary}\label{fiber-quasi-abelian}
Let  $\mathfrak{g}_\bullet=\{\mathfrak{g}_0\xrightarrow{f}\mathfrak{g}_1\}$ be a morphism of differential graded Lie algebras. If $f$ is injective in cohomology, then its homotopy fiber $hofib(\mathfrak{g}_\bullet)$ is homotopy abelian.
\end{corollary}

\begin{proof} 
By functoriality of $TW$, the commutative diagram 
\[ \xymatrix{0\ar[r]&0\ar[r]\ar[d]&\mathfrak{g}_0\ar[r]^{\Id_{\mathfrak{g}_0}}\ar[d]^f&\mathfrak{g}_0\ar[r]\ar[d]&0\\
0\ar[r]&\mathfrak{g}_1\ar[r]^{\Id_{\mathfrak{g}_1}}&\mathfrak{g}_1\ar[r]&0\ar[r]&0}\]
induces a short exact sequence of DG-Lie algebras
\[ 0\to \Omega \mathfrak{g}_1\to TW(\mathfrak{g}_\bullet)\to \mathfrak{g}_0\to 0,\]
and a straightforward diagram chasing gives that 
$f$ is injective in cohomology if and only if 
$\Omega\mathfrak{g}_1\to TW(\mathfrak{g}_\bullet)$ is surjective in cohomology. By Corollary~\ref{loop}, $\Omega \mathfrak{g}_1$ is homotopy abelian, and so 
Lemma~\ref{prop.iniettivoincohomologiaversohomotopyabelian} implies that $TW(\mathfrak{g}_\bullet)$ is homotopy abelian, too.
\end{proof}

\bigskip
\section{The double homotopy fiber of a commutative square of DG-Lie algebras}

We now turn our attention to commutative squares of DG-Lie algebras, i.e., to the category $Fun(\Delta^1\times \Delta^1,\mathbf{DGLA})$.

\begin{definition}\label{def.doublefiber} 
Let
\[ \tilde{\mathfrak{g}}_{\bullet\bullet}:\qquad  \xymatrix{\tilde{\mathfrak{g}}_{00}\ar[d]^{}\ar[r]^{}&\tilde{\mathfrak{g}}_{01}\ar[d]^{}\\
\tilde{\mathfrak{g}}_{10}\ar[r]^{}&\tilde{\mathfrak{g}}_{11}}\]
be a fibrant object in $Fun(\Delta^1\times \Delta^1,\mathbf{DGLA})$. We 
write $fib^{[2]}(\tilde{\mathfrak{g}}_{\bullet\bullet})$, and call it the \emph{double fiber} of $\tilde{\mathfrak{g}}_{\bullet\bullet}$, for the DG-Lie algebra given by 
\[fib^{[2]}(\tilde{\mathfrak{g}}_{\bullet\bullet}):=\ker\left(\tilde{\mathfrak{g}}_{00}\xrightarrow{}\tilde{\mathfrak{g}}_{10}\times_{\tilde{\mathfrak{g}}_{11}}\tilde{\mathfrak{g}}_{01}\right)\,.\]
\end{definition}

\begin{remark}\label{rem.quasi.iso}
It is immediate to see that, in the situation of  Definition~\ref{def.doublefiber}, 
there exist two short exact sequences of DG-Lie algebras
\[ 0\to fib^{[2]}(\tilde{\mathfrak{g}}_{\bullet\bullet})\to 
\ker(\tilde{\mathfrak{g}}_{0,0}\to \tilde{\mathfrak{g}}_{1,0})\to 
\ker(\tilde{\mathfrak{g}}_{1,0}\to \tilde{\mathfrak{g}}_{1,1})\to 0,\]
\[ 0\to fib^{[2]}(\tilde{\mathfrak{g}}_{\bullet\bullet})\to 
\ker(\tilde{\mathfrak{g}}_{0,0}\to \tilde{\mathfrak{g}}_{0,1})\to 
\ker(\tilde{\mathfrak{g}}_{0,1}\to \tilde{\mathfrak{g}}_{1,1})\to 0.\]
In particular, every weak equivalence $\tilde{\mathfrak{g}}_{\bullet\bullet}\to \tilde{\mathfrak{h}}_{\bullet\bullet}$ of fibrant commutative squares induces a quasi-isomorphism 
$fib^{[2]}(\tilde{\mathfrak{g}}_{\bullet\bullet})\to fib^{[2]}(\tilde{\mathfrak{h}}_{\bullet\bullet})$; it is worth to notice that the same fact can be proved in a more abstract 
way by applying twice the coglueing theorem.
\end{remark}

In analogy with the case of morphisms, we define the \emph{double homotopy fiber} of a commutative square 
$\mathfrak{g}_{\bullet\bullet}\colon \Delta^1\times\Delta^1\to  \mathbf{DGLA}$.
\begin{definition}
Let $\mathfrak{g}_{\bullet\bullet}\in Fun(\Delta^1\times\Delta^1,\mathbf{DGLA})$. We set \[
hofib^{[2]}(\mathfrak{g}_{\bullet\bullet}):=fib^{[2]}(\tilde{\mathfrak{g}}_{\bullet\bullet}),
\]
where $\mathfrak{g}_{\bullet\bullet}\to \tilde{\mathfrak{g}}_{\bullet\bullet}$ is a fibrant resolution of $\mathfrak{g}_{\bullet\bullet}$ in the inverse Reedy structure, i.e., a weak equivalence, with $\tilde{\mathfrak{g}}_{\bullet\bullet}$ a fibrant commutative diagram. We call (the homotopy class of) $hofib^{[2]}(\mathfrak{g}_{\bullet\bullet})$ the double homotopy fiber of $\mathfrak{g}_{\bullet\bullet}$.
\end{definition}

The proof that the  double homotopy fiber is well defined in the homotopy category may be done by standard model category methods.  
For later use, we prove this fact by describing a canonical fibrant resolution for every commutative square $\mathfrak{g}_{\bullet\bullet}$.
 
\begin{definition}\label{def.twdoublefiber} 
 Given a commutative square of DG-Lie algebras
\[ \mathfrak{g}_{\bullet\bullet}:\qquad  \xymatrix{\mathfrak{g}_{00}\ar[d]^{}\ar[r]^{}&\mathfrak{g}_{01}\ar[d]^{}\\
\mathfrak{g}_{10}\ar[r]^{}&\mathfrak{g}_{11}}\]
its Thom-Whitney fibrant resolution $\mathfrak{g}_{\bullet\bullet}\xrightarrow{\sim} \mathfrak{g}_{TW;\bullet\bullet}\twoheadrightarrow 0$ is 
\[ \mathfrak{g}_{TW;\bullet\bullet}:\qquad  \xymatrix{\mathfrak{g}_{TW;00}\ar[d]^{}\ar[r]^{}&\mathfrak{g}_{TW;01}\ar[d]^{}\\
\mathfrak{g}_{TW;10}\ar[r]^{}&\mathfrak{g}_{TW;11}}\]
where 
\[
\mathfrak{g}_{TW;11}=\mathfrak{g}_{11};
\qquad
\mathfrak{g}_{TW;01}=\mathfrak{g}_{01}\times^{\vert_1}_{\mathfrak{g}_{11}}\Omega^\bullet(\Delta^1;\mathfrak{g}_{11})
\qquad
\mathfrak{g}_{TW;10}=\mathfrak{g}_{10}\times^{\vert_1}_{\mathfrak{g}_{11}}\Omega^\bullet(\Delta^1;\mathfrak{g}_{11})
\]
and $\mathfrak{g}_{TW;00}$ is given by the  fiber product
\[\mathfrak{g}_{00}\times^{\vert_1\times\vert_1}_{\mathfrak{g}_{01}\times\mathfrak{g}_{10}} \left(\Omega^\bullet(\Delta^1;\mathfrak{g}_{01})\times \Omega^\bullet(\Delta^1;\mathfrak{g}_{10})\right)
\times_{\Omega^\bullet(\Delta^1;\mathfrak{g}_{11})\times \Omega^\bullet(\Delta^1;\mathfrak{g}_{11})}^{\vert_{1\bullet}\times \vert_{\bullet1}} \Omega^\bullet(\Delta^1\times\Delta^1;\mathfrak{g}_{11}),
\]
where
\[
\vert_{1\bullet}, \vert_{\bullet1}\colon \Omega^\bullet(\Delta^1\times\Delta^1;\mathfrak{g}_{11})\to  \Omega^\bullet(\Delta^1;\mathfrak{g}_{11})
\]
are the pullbacks along the inclusions $\Delta^1\cong \{1\}\times \Delta^1\hookrightarrow \Delta^1\times \Delta^1$ and 
 $\Delta^1\cong \Delta^1\times\{1\}\hookrightarrow \Delta^1\times \Delta^1$, respectively.
The morphisms $\mathfrak{g}_{TW;01}\to \mathfrak{g}_{TW;11}$ and $\mathfrak{g}_{TW;10}\to \mathfrak{g}_{TW;11}$ are given by the compositions
\[
\mathfrak{g}_{TW;01}\hookrightarrow \Omega^\bullet(\Delta^1;\mathfrak{g}_{11})\xrightarrow{\vert_0}\mathfrak{g}_{11}
\]
and
\[
\mathfrak{g}_{TW;10}\hookrightarrow \Omega^\bullet(\Delta^1;\mathfrak{g}_{11})\xrightarrow{\vert_0}\mathfrak{g}_{11}
\]
respectively. Finally, the morphisms $\mathfrak{g}_{TW;00}\to \mathfrak{g}_{TW;01}$ and $\mathfrak{g}_{TW;00}\to \mathfrak{g}_{TW;10}$
are given by the compositions
\[
\mathfrak{g}_{TW;00}\hookrightarrow \Omega^\bullet(\Delta^1;\mathfrak{g}_{01})
\times_{\Omega^\bullet(\Delta^1;\mathfrak{g}_{11})}^{\vert_{1\bullet}} \Omega^\bullet(\Delta^1\times\Delta^1;\mathfrak{g}_{11})
\xrightarrow{(\vert_0,\vert_{0\bullet})}\mathfrak{g}_{TW;01}
\]
and
\[
\mathfrak{g}_{TW;00}\hookrightarrow \Omega^\bullet(\Delta^1;\mathfrak{g}_{10})
\times_{\Omega^\bullet(\Delta^1;\mathfrak{g}_{11})}^{\vert_{\bullet1}} \Omega^\bullet(\Delta^1\times\Delta^1;\mathfrak{g}_{11})
\xrightarrow{(\vert_0,\vert_{\bullet0})}\mathfrak{g}_{TW;10}
\]
The morphism $\mathfrak{g}_{\bullet\bullet}\xrightarrow{\sim} \mathfrak{g}_{TW;\bullet\bullet}$ is the obvious componentwise inclusion.
\end{definition}

\begin{definition}
We shall denote by $TW^{[2]}(\mathfrak{g}_{\bullet\bullet}):=fib^{[2]}(\mathfrak{g}_{TW;\bullet\bullet})$ the double fiber of the Thom-Whitney fibrant resolution, and will call it the Thom-Whitney model for the double homotopy fiber of $\mathfrak{g}_{\bullet\bullet}$.
\end{definition} 

\begin{remark}
By unwinding the definition, one sees that, if $\mathfrak{g}_{\bullet\bullet}$ is the commutative square
\[ \xymatrix{\mathfrak{g}_{00}\ar[d]^{v_0}\ar[r]^{h_0}&\mathfrak{g}_{01}\ar[d]^{v_1}\\
\mathfrak{g}_{10}\ar[r]^{h_1}&\mathfrak{g}_{11}}\]
of DG-Lie algebras, then  $TW^{[2]}(\mathfrak{g}_{\bullet\bullet})$ is the DG-Lie subalgebra  of 
\[\mathfrak{g}_{00}\times \Omega^\bullet(\Delta^1;\mathfrak{g}_{10})\times \Omega^\bullet(\Delta^1;\mathfrak{g}_{01})\times \Omega^\bullet(\Delta^1\times\Delta^1;\mathfrak{g}_{11})\] 
consisting of those quadruples $(\omega_{00},\omega_{01},\omega_{10},\omega_{11})$ such that
\[
\begin{cases}
\omega_{01}\bigr\vert_{0}=0;\qquad \omega_{01}\bigr\vert_{1}=h_0(\omega_{00})\\
\omega_{10}\bigr\vert_{0}=0;\qquad \omega_{10}\bigr\vert_{1}=v_0(\omega_{00})\\
\omega_{11}\bigr\vert_{0\bullet}=0;\qquad \omega_{11}\bigr\vert_{1\bullet}=h_1(\omega_{10})\\
\omega_{11}\bigr\vert_{\bullet0}=0;\qquad \omega_{11}\bigr\vert_{\bullet1}=v_1(\omega_{01})
\end{cases}
\]
\end{remark} 

\begin{remark}\label{rem.doppiafibra}
It is straightforward to observe that there exists a natural isomorphism 
\[TW^{[2]}(\mathfrak{g}_{\bullet\bullet})\cong TW\left(TW(\mathfrak{g}_{\bullet0})\to TW(\mathfrak{g}_{\bullet1})\right)\cong TW\left(TW(\mathfrak{g}_{0\bullet})\to TW(\mathfrak{g}_{1\bullet})\right),\]
where 
\[TW(\mathfrak{g}_{\bullet0})\to TW(\mathfrak{g}_{\bullet1}),\qquad TW(\mathfrak{g}_{0\bullet})\to TW(\mathfrak{g}_{1\bullet})\]
are induced by the commutative square $\mathfrak{g}_{\bullet\bullet}$ via the functoriality of the Thom-Whitney model for homotopy fibers.
\end{remark}

\begin{lemma}\label{to-be-used-immediately}
Let $\tilde{\mathfrak{g}}_{\bullet\bullet}$ be a fibrant commutative square of DG-Lie algebras. 
There exists a natural  quasi-isomorphism of differential graded Lie algebras
\[ fib^{[2]}(\tilde{\mathfrak{g}}_{\bullet\bullet})\to TW^{[2]}(\tilde{\mathfrak{g}}_{\bullet\bullet}),\qquad \tilde{\omega}_{00}\mapsto (\tilde{\omega}_{00},0,0,0),\]
\end{lemma}
\begin{proof}
Since $\tilde{\mathfrak{g}}_{\bullet\bullet}$ is fibrant, the resolution $\tilde{\mathfrak{g}}_{\bullet\bullet}\xrightarrow{\sim} \tilde{\mathfrak{g}}_{TW;\bullet\bullet}$ is a weak equivalence of fibrant commutative squares of DG-Lie algebras, and so
\[
 fib^{[2]}(\tilde{\mathfrak{g}}_{\bullet\bullet})\to fib^{[2]}(\tilde{\mathfrak{g}}_{TW;\bullet\bullet})=TW^{[2]}(\tilde{\mathfrak{g}}_{\bullet\bullet})
\]
is a quasi-isomorphism by Remark \ref{rem.quasi.iso}. Written out explicitly, it is immediate to see that the morphism $fib^{[2]}(\tilde{\mathfrak{g}}_{\bullet\bullet})\to TW^{[2]}(\tilde{\mathfrak{g}}_{\bullet\bullet})$ induced this way is $\tilde{\omega}_{00}\mapsto (\tilde{\omega}_{00},0,0,0)$.
\end{proof}

Now, the proof that the double homotopy fiber is well defined is immediate. More precisely we have
\begin{proposition}\label{homotopy-invariance-of-double-fiber}
Let $\mathfrak{g}_{\bullet\bullet}$ be a commutative square of DG-Lie algebras. Then, for any fibrant resolution  $\mathfrak{g}_{\bullet\bullet}\to \tilde{\mathfrak{g}}_{\bullet\bullet}$ of $\mathfrak{g}_{\bullet\bullet}$, the DG-Lie algebra $fib^{[2]}(\tilde{\mathfrak{g}}_{\bullet\bullet})$ is weakly equivalent to $TW^{[2]}(\mathfrak{g}_{\bullet\bullet})$. In particular, the homotopy equivalence class of $hofib^{[2]}(\mathfrak{g}_{\bullet\bullet})$ is well defined and only depends on the homotopy equivalence class of $\mathfrak{g}_{\bullet\bullet}$.
\end{proposition}
\begin{proof}
The weak equivalence $\mathfrak{g}_{\bullet\bullet}\to \tilde{\mathfrak{g}}_{\bullet\bullet}$  of commutative squares of DG-Lie algebras induces by functoriality 
a weak equivalence $\mathfrak{g}_{TW;\bullet\bullet}\xrightarrow{\sim} \tilde{\mathfrak{g}}_{TW;\bullet\bullet}$. By Remark \ref{rem.quasi.iso} this gives a quasi-isomorphism $TW^{[2]}(\mathfrak{g}_{\bullet\bullet})\xrightarrow{\sim} TW^{[2]}(\tilde{\mathfrak{g}}_{\bullet\bullet})$, while by Lemma \ref{to-be-used-immediately} we have a canonical quasi-iso\-morphism $ fib^{[2]}(\tilde{\mathfrak{g}}_{\bullet\bullet})\to TW^{[2]}(\tilde{\mathfrak{g}}_{\bullet\bullet})$.
\end{proof}

\begin{lemma}\label{lem.hypercohomology} 
The cohomology of the double homotopy fiber of a commutative square $\mathfrak{g}_{\bullet\bullet}$ is naturally isomorphic to the hypercohomology of the complex of DG-vector spaces
\[ \mathrm{Tot}(\mathfrak{g}_{\bullet\bullet})\colon\qquad  \mathfrak{g}_{00}\to \mathfrak{g}_{01}\oplus \mathfrak{g}_{10}\to \mathfrak{g}_{11}.\]  
\end{lemma}

\begin{proof}
A fibrant resolution $\mathfrak{g}_{\bullet\bullet}\xrightarrow{\sim} \tilde{\mathfrak{g}}_{\bullet\bullet}$ induces a quasi-isomorphism of complexes of DG-vector spaces $\mathrm{Tot}(\mathfrak{g}_{\bullet\bullet})\xrightarrow{\sim} \mathrm{Tot}(\tilde{\mathfrak{g}}_{\bullet\bullet})$. Since $\tilde{\mathfrak{g}}_{\bullet\bullet}$ is fibrant, we have a short exact sequence of DG-vector spaces
\[ 0\to fib^{[2]}(\tilde{\mathfrak{g}}_{\bullet\bullet})\to \tilde{\mathfrak{g}}_{00}\to \tilde{\mathfrak{g}}_{01}\oplus \tilde{\mathfrak{g}}_{10}\to \tilde{\mathfrak{g}}_{11}\to 0,\]  
and so
\[
\mathbb{H}^i(\mathrm{Tot}(\mathfrak{g}_{\bullet\bullet}))\cong \mathbb{H}^i(\mathrm{Tot}(\tilde{\mathfrak{g}}_{\bullet\bullet}))=H^i(fib^{[2]}(\tilde{\mathfrak{g}}_{\bullet\bullet}))=H^i(hofib^{[2]}(\mathfrak{g}_{\bullet\bullet})).
\]
\end{proof}

\begin{corollary}\label{cor.cohomology-minus-2}
Assume that a commutative square $\mathfrak{g}_{\bullet\bullet}$ of DG-Lie algebras is the fiber product of two inclusions of DG-Lie algebras. Then
\[
H^*(hofib^{[2]}(\mathfrak{g}_{\bullet\bullet}))\cong H^{*-2}(\mathfrak{g}_{11}/(\mathfrak{g}_{10}+\mathfrak{g}_{01})).
\]
Moreover, the isomorphism $H^*(TW^{[2]}(\mathfrak{g}_{\bullet\bullet}))\cong H^{*-2}(\mathfrak{g}_{11}/(\mathfrak{g}_{01}+\mathfrak{g}_{10}))$ is induced by the morphism of chain complexes
\begin{align*}
TW^{[2]}(\mathfrak{g}_{\bullet\bullet})&\to \mathfrak{g}_{11}/(\mathfrak{g}_{10}+\mathfrak{g}_{01})[-2]\\
(\omega_{00},\omega_{10},\omega_{01},\omega_{11})&\mapsto \int_{\Delta^1\times\Delta^1}\omega_{11} \mod \mathfrak{g}_{01}+\mathfrak{g}_{10}).
\end{align*}
\end{corollary}

\begin{proof}
The first part is an immediate consequence of the exact sequence of DG-vector spaces
\[ 0\to \mathfrak{g}_{00}\to \mathfrak{g}_{01}\oplus \mathfrak{g}_{10}\to 
\mathfrak{g}_{11}\to \frac{\mathfrak{g}_{11}}{\;\mathfrak{g}_{10}+\mathfrak{g}_{01}\;}\to 0\,.\]  
The second part follows from Lemma~\ref{lemma.tw-quotient}.
\end{proof}

\begin{corollary}\label{cor.abelinitafibraquadrato}
Let  $\mathfrak{g}_{\bullet\bullet}$
be a commutative square of differential graded Lie algebras.
If the induced map 
\[ H^*(\mathfrak{g}_{00})\to H^*(\mathfrak{g}_{10})\times_{H^*(\mathfrak{g}_{11})}H^*(\mathfrak{g}_{01})\]
is injective, then the double homotopy fiber $hofib^{[2]}(\mathfrak{g}_{\bullet\bullet})$ is homotopy abelian.
In particular, if either $H^*(\mathfrak{g}_{00})\to H^*(\mathfrak{g}_{10})$ or 
$H^*(\mathfrak{g}_{00})\to H^*(\mathfrak{g}_{01})$ is injective, then $hofib^{[2]}(\mathfrak{g}_{\bullet\bullet})$ is homotopy abelian.
\end{corollary}

\begin{proof} It is not restrictive to assume that $\mathfrak{g}_{\bullet\bullet}$ is fibrant diagram. Then 
we have a factorization 
\[ H^*(\mathfrak{g}_{00})\to H^*\left(\mathfrak{g}_{10}\times_{\mathfrak{g}_{11}}\mathfrak{g}_{01}\right)
\to H^*(\mathfrak{g}_{10})\times_{H^*(\mathfrak{g}_{11})}H^*(\mathfrak{g}_{01})\]
and the conclusion follows by Corollary~\ref{fiber-quasi-abelian}.
\end{proof}

\begin{corollary} Let $\mathfrak{g}_0\xrightarrow{f}\mathfrak{g}_1\xrightarrow{g}\mathfrak{g}_2$ be morphisms of differential graded Lie algebras.
Then the homotopy fiber of the natural  morphism
$TW(\mathfrak{g}_0\xrightarrow{f}\mathfrak{g}_1)\to TW(\mathfrak{g}_0\xrightarrow{gf}\mathfrak{g}_2)$ 
is homotopy abelian.
\end{corollary}

\begin{proof} By Remark~\ref{rem.doppiafibra} the homotopy fiber of $TW(\mathfrak{g}_0\xrightarrow{f}\mathfrak{g}_1)\to TW(\mathfrak{g}_0\xrightarrow{gf}\mathfrak{g}_2)$ is quasi-isomorphic to the double homotopy fiber of the diagram
\[ \xymatrix{\mathfrak{g}_0\ar[d]_{\Id_{\mathfrak{g}_0}}\ar[r]^{f}&\mathfrak{g}_1\ar[d]^{g}\\
\mathfrak{g}_0\ar[r]^{gf}&\mathfrak{g}_2\,,}\]
which is homotopy abelian by Corollary~\ref{cor.abelinitafibraquadrato}.
\end{proof}

\bigskip
\section{Cartan calculi and their induced morphisms of deformation functors}
\label{sec.cartanandperiods}

We begin by recalling a few basic facts on the general theory of infinitesimal deformation functors that we will need in what follows: a complete account can be found in \cite{GoMil1,ManettiSemireg,ManettiSeattle}.
We shall denote by $\Set$ the category of sets, by 
$\mathbf{Grp}$ the category of groups, by $\Art$ the category of local Artin $\K$-algebras with residue field $\K$.
For every local ring $A$ its maximal ideal will be denoted $\mathfrak{m}_A$.  For a Lie algebra $\mathfrak{g}^0$, its exponential functor is 
\[ \exp_{\mathfrak{g}^0}\colon \Art\to \mathbf{Grp},\qquad \exp_{\mathfrak{g}^0}(A)=\exp(\mathfrak{g}^0\otimes\mathfrak{m}_A)\;.\]
Given a DG-Lie algebra $\mathfrak{g}$, its Maurer-Cartan functor is
\[ \MC_{\mathfrak{g}}\colon \Art\to\Set,\qquad \MC_{\mathfrak{g}}(A)=\{x^1\in \mathfrak{g}^1\otimes\mathfrak{m}_A\mid dx^1+\frac{1}{2}[x^1,x^1]=0\},\]
where $\mathfrak{g}^i$ denotes the subspace of homogeneous elements of $\mathfrak{g}$ of degree $i$.
The (left) gauge action of $\exp_{\mathfrak{g}^0}$ on $\MC_\mathfrak{g}$ may be written as
\[ e^{x^0}\ast x^1=x^1+\sum_{n\ge 0}\frac{[x^0,-]^n}{(n+1)!}([x^0,x^1]-dx^0)\]
and the corresponding quotient is called the deformation functor associated to $\mathfrak{g}$: 
\[ \Def_\mathfrak{g}=\frac{\;\MC_\mathfrak{g}\;}{\exp_{\mathfrak{g}^0}}\,.\]
 The basic theorem of deformation theory asserts that every quasi-isomorphism $\mathfrak{g}_0\to \mathfrak{g}_1$ of DG-Lie algebras induces an isomorphism of functors $\Def_{\mathfrak{g}_1}\to \Def_{\mathfrak{g}_2}$. 
 
As a consequence, up to isomorphism, the deformation functor $\Def_{\mathfrak{g}}$ only depends on the homotopy class of the DG-Lie algebra $\mathfrak{g}$. 
In particular, it is meaningful to consider the deformation functor associated with the homotopy fiber of a morphism $\mathfrak{g}_\bullet=(\mathfrak{g}_0\xrightarrow{f} \mathfrak{g}_1)$ of DG-Lie algebras. Actually, one has the following useful concrete description for the deformation functor associated to $hofib(\mathfrak{g}_\bullet)$. First, one defines the Maurer-Cartan functor: 
\[ \MC_{\mathfrak{g}_\bullet}(A)=\{(x_0^1,e^{x_1^0})\in 
\left\{(\mathfrak{g}_0^1\otimes\mathfrak{m}_A)\times \exp({\mathfrak{g}_1}^0\otimes\mathfrak{m}_A)\mid 
dx^1_0+\frac{1}{2}[x^1_0,x^1_0]=0,\; e^{x^0_1}\ast 0=f(x^1_0)\right\}.\]
Notice that we write  systematically  $x_i^j$ in order to denote a degree $j$ element in the DG-Lie algebra $\mathfrak{g}_i$.
The gauge action of $ \exp_{{\mathfrak{g}_0}^0}$ on $\MC_{\mathfrak{g}_0}$ lifts to a (left) action of the group functor $\exp_{{\mathfrak{g}_0}^0}\times\exp_{d{\mathfrak{g}_1}^{-1}}$ 
on $\MC_{\mathfrak{g}_\bullet}$ by setting
\[ (e^{x_0^0},e^{dx_1^{-1}})\ast (x_0^1,e^{x_1^0})=(e^{x_0^0}\ast x_0^1,\, e^{f(x_0^0)}e^{x_1^0}e^{-dx_1^{-1}}),\qquad (x_0^0,x_1^{-1})\in {\mathfrak{g}_0}^0\oplus {\mathfrak{g}_1}^{-1}.\]
 
\begin{lemma}\label{lem.modellopiccolofibraomotopica}
Given a 
morphism of differential graded Lie algebras $\mathfrak{g}_\bullet$,  the functor of Artin rings  
\[ \Def_{\mathfrak{g}_\bullet}=\frac{\MC_{\mathfrak{g}_\bullet}}{\exp_{{\mathfrak{g}_0}^0}\times\exp_{d{\mathfrak{g}_1}^{-1}}}\;.\]
is  isomorphic to the deformation functor associated with the homotopy fiber of $\mathfrak{g}_\bullet$.
More precisely,  every commutative diagram of differential graded Lie algebras of the form
\[ \xymatrix{\mathfrak{g}_0\ar[d]\ar[r]^{\sim}&\tilde{\mathfrak{g}}_0\ar@{->>}[d]&\ker(\tilde{\mathfrak{g}}_0\to \tilde{\mathfrak{g}}_1)\ar@{^{ )}->}[l]\ar[d]\\
\mathfrak{g}_1\ar[r]^{\sim}&\tilde{\mathfrak{g}}_1&0\ar[l]}\]
induces canonically two isomorphisms of functors
\[ \Def_{\mathfrak{g}_\bullet}\xrightarrow{\;\cong\;}\Def_{\tilde{\mathfrak{g}}_\bullet}\xleftarrow{\;\cong\;}%\Def_h=
\Def_{\ker(\tilde{\mathfrak{g}}_0\to \tilde{\mathfrak{g}}_1)}\;.\]
In particular, the isomorphism $\Def_{\mathfrak{g}_\bullet}\cong \Def_{TW(\mathfrak{g}_\bullet)}$ is induced, at the level of Maurer-Cartan elements, by the natural transformation
\[ \MC_{\mathfrak{g}_\bullet}\to \MC_{TW(\mathfrak{g}_\bullet)},\qquad (x_0^1,e^{x_1^0})\mapsto (x_0^1,e^{tx_1^0}\ast 0)\in TW(\mathfrak{g}_\bullet)\subseteq \mathfrak{g}_0\times \Omega^\bullet(\Delta^1;\mathfrak{g}_1)\;.\] 
\end{lemma}

\begin{proof} This is proved for instance in \cite[Thm. 2]{FMcone} via homotopy transfer of $L_{\infty}$ structures, and in  
\cite[Thm. 6.14]{ManettiSemireg} via implicit function theorem for extended deformation functors.
A more explicit and direct proof will appear  in a forthcoming book\footnote{M. Manetti: \emph{Lie methods in deformation theory,} in preparation.} by the second author.\end{proof}

\begin{remark}\label{rem.injective}
When $\mathfrak{g}_\bullet=(\mathfrak{g}_0\to 0)$ is the terminal morphism, then  $\MC_{\mathfrak{g}_\bullet}$ and $\Def_{\mathfrak{g}_\bullet}$ reduce to the usual functors $\MC_{\mathfrak{g}_0}$ and $\Def_{\mathfrak{g}_0}$, respectively. When  $\mathfrak{g}_\bullet=(\mathfrak{g}_0\xrightarrow{f} \mathfrak{g}_1)$ is an injective morphism of DG-Lie algebras, the Maurer-Cartan functor $\MC_{\mathfrak{g}_\bullet}$ admits the simpler description 
\[ \MC_{\mathfrak{g}_\bullet}(A)=\{e^{x_1^0}\in \exp(\mathfrak{g}_1^0\otimes\mathfrak{m}_A)\mid 
e^{x_1^0}\ast 0\in f(\mathfrak{g}_0^1)\otimes \mathfrak{m}_A\}.\]
  
 \end{remark}

\begin{definition}
Let $ \mathfrak{g}$ be a differential graded Lie algebra and let $\flat\colon \mathfrak{g}\to \mathfrak{g}[1]$ be the tautological 
linear map of degree $-1$. The mapping cone  of the identity of $\mathfrak{g}$ is the graded vector space 
\[\cone(\mathrm{Id}_\mathfrak{g})= \mathfrak{g}\oplus  \mathfrak{g}[1]\]
endowed with the unique structure of 
differential graded Lie algebra such that the natural inclusion $\mathfrak{g}\hookrightarrow \cone(\mathrm{Id}_\mathfrak{g})$  
is a morphism of DG-Lie algebras and such that for every $x,y\in \mathfrak{g}$ we have: 
\[  d(\flat x)=x-\flat dx,\qquad [\flat x,y]=\flat[x,y],\qquad [\flat x,\flat y]=0\;.\]
\end{definition}

\begin{lemma}\label{isomorphisms.def}
Let $ \mathfrak{g}$ be a DG-Lie algebra. Then there exists a natural acyclic fibration 
\[
TW(\mathfrak{g}_{}\hookrightarrow\cone(\mathrm{Id}_{\mathfrak{g}_{}}))\stackrel{\sim}{\twoheadrightarrow}\mathfrak{g},
\] inducing 
isomorphisms of deformation functors 
\[\Def_{\mathfrak{g}}\cong  \Def_{TW(\mathfrak{g}\hookrightarrow \cone(\mathrm{Id}_\mathfrak{g}))}\cong \Def_{(\mathfrak{g}\hookrightarrow \cone(\mathrm{Id}_\mathfrak{g}))}\,.\] 
These isomorphism are induced by the morphisms of Maurer-Cartan functors
\[ \MC_{(\mathfrak{g}\hookrightarrow \cone(\mathrm{Id}_\mathfrak{g}))}\to \MC_{\mathfrak{g}},\qquad 
(x,e^{\flat y})\mapsto x,
\]
\[ \MC_{(\mathfrak{g}\hookrightarrow \cone(\mathrm{Id}_\mathfrak{g}))}\to \MC_{TW(\mathfrak{g}\hookrightarrow \cone(\mathrm{Id}_\mathfrak{g}))},\qquad 
 (x,e^{-\flat y})\mapsto (x,e^{-t\flat y}\ast 0)\,.
\]
\end{lemma}

\begin{proof}
Since the mapping cone of the identity is acyclic,  the natural  
surjective map of  differential graded Lie algebras
\[ TW(\mathfrak{g}_{}\hookrightarrow\cone(\mathrm{Id}_{\mathfrak{g}_{}}))\twoheadrightarrow TW(\mathfrak{g}_{}\to 0)=\mathfrak{g}_{}\]
is a quasi-isomorphism. One concludes by Lemma \ref{lem.modellopiccolofibraomotopica}.
\end{proof}

\begin{remark}\label{rem.g-to-TW}
At the level of chain complexes, a homotopy inverse of the quasi-isomorphism $TW(\mathfrak{g}_{}\hookrightarrow\cone(\mathrm{Id}_{\mathfrak{g}_{}}))\xrightarrow{\sim} \mathfrak{g}$ is the quasi-isomorphism of complexes
\begin{align*}
\mathfrak{g}&\xrightarrow{\sim} TW(\mathfrak{g}_{}\hookrightarrow\cone(\mathrm{Id}_{\mathfrak{g}_{}}))\\
x&\mapsto (x,t\,x +dt\, \flat x).
\end{align*}
\end{remark}

\begin{lemma}\label{lem.sezionecartan} 
Let $ \mathfrak{g}$ be a DG-Lie algebra. The inverse of the isomorphism of deformation functors $\Def_{(\mathfrak{g}\hookrightarrow \cone(\mathrm{Id}_\mathfrak{g}))}\to \Def_{\mathfrak{g}}$ from Lemma \ref{isomorphisms.def} is induced by the morphism of Maurer-Cartan functors
\[ \MC_{\mathfrak{g}}\to \MC_{(\mathfrak{g}\hookrightarrow \cone(\mathrm{Id}_\mathfrak{g}))},\qquad 
x\mapsto (x,e^{-\flat x})\,.
\]
\end{lemma}
\begin{proof} We only need to show that $x\mapsto (x,e^{-\flat x})$ is actually a map  from $\MC_{\mathfrak{g}}$ to $\MC_{(\mathfrak{g}\hookrightarrow \cone(\mathrm{Id}_\mathfrak{g}))}$, i.e., that  $e^{-\flat x}\ast 0=x$ whenever $x\in \MC_{\mathfrak{g}}$. Once this is shown, the statement follows by noticing that this map is a right inverse to the map $\MC_{(\mathfrak{g}\hookrightarrow \cone(\mathrm{Id}_\mathfrak{g}))}\to \MC_{\mathfrak{g}}$ from Lemma \ref{isomorphisms.def}.
By definition of the gauge action, and since $[\flat x_1,[\flat x_2,x_3]]=[\flat x_1,\flat[x_2,x_3]]=0$ for any $x_1,x_2,x_3\in \mathfrak{g}$, we have
\[\begin{split}
e^{-\flat x}*0&=\sum_{n=0}^\infty \frac{[-\flat x, -]^n}{(n+1)!} (d(\flat x))=
\sum_{n=0}^1 \frac{[-\flat x, -]^n}{(n+1)!}\ (x-\flat dx)\\
&=x-\flat\left(\frac{1}{2}[x,x]+dx\right)=x\;.\end{split}\]
\end{proof}

\begin{remark} By general facts about $L_\infty$-algebras, there exists  $L_\infty$ right inverses of the morphism 
$TW(\mathfrak{g}_{}\hookrightarrow\cone(\mathrm{Id}_{\mathfrak{g}_{}}))\twoheadrightarrow \mathfrak{g}_{}$
 of Lemma~\ref{isomorphisms.def}. It is not difficult to prove that there exists a  right inverse that induces the morphism of Lemma~\ref{lem.sezionecartan} between Maurer-Cartan functors.
\end{remark}

\begin{definition}[\cite{FMperiods,algebraicBTT}]\label{def.cartanhomotopy}
Let $\mathfrak{g}_0$ and $\mathfrak{g}_1$ be two differential graded Lie algebras. An element 
\[ \bi\in \Hom_{\K}^{-1}(\mathfrak{g}_0,\mathfrak{g}_1),  \]
i.e., a linear map $\bi\colon \mathfrak{g}_0\to\mathfrak{g}_1$ of degree $-1$, is called a \emph{Cartan homotopy} if, for every $x,y\in \mathfrak{g}_0$, we
have
\[[\bi_x,\bi_{y}]_{\mathfrak{g}_1}=0,\qquad \bi_{[x,y]_{\mathfrak{g}_0}}=[\bi_x,d_{\mathfrak{g}_1}\bi_y].\]
The \emph{boundary} of a Cartan homotopy $\bi$ is the morphism of DG-vector spaces
\[ \bl=d_{\mathfrak{g}_1} \bi+\bi  d_{\mathfrak{g}_0}\colon \mathfrak{g}_0\to \mathfrak{g}_1.\]
\end{definition}
Clearly the boundary of a Cartan homotopy is null homotopic as a map of cochain complexes. Moreover, the Cartan identities satisfied by $\bi$ imply that its boundary $\bl$ is actually a morphism of DG-Lie algebras.
As the notation suggests, the prototypical example of a Cartan homotopy is the contraction of differential forms with vector fields on a smooth manifold \cite{cartan50}; 
the boundary morphism is the Lie derivative in this case. We refer to  \cite[Section 7]{IaconoDP} for a more exotic example in the framework of differential Batalin-Vilkovisky algebras.

\begin{remark}\label{rem.homotopy implica homotopyX A}
Let $\bi\colon \mathfrak{g}_0\to \mathfrak{g}_1[-1]$ be a Cartan homotopy with  boundary $\bl$ and let $\Omega$ be a
differential graded commutative algebra.  Then the maps
\begin{align*}
\bi\otimes \mathrm{Id}_\Omega\colon  \mathfrak{g}_0\otimes \Omega\to (\mathfrak{g}_1\otimes \Omega)[-1],&\qquad
(x \otimes \omega)\mapsto\bi_x \otimes \omega,\\
\mathrm{Id}_\Omega\otimes \bi\colon  \Omega\otimes \mathfrak{g}_0\to (\Omega\otimes \mathfrak{g}_1)[-1],&\qquad
(\omega \otimes x)\mapsto (-1)^{\bar{\omega}}\omega\otimes \bi_x,
\end{align*}
are  Cartan homotopies with boundaries $\bl\otimes \mathrm{Id}_\Omega$ and $\mathrm{Id}_\Omega\otimes \bl$ respectively. 
\end{remark}

\begin{lemma}\label{lem.universalCartanhomotopy} 
Let $\mathfrak{g}_0,\mathfrak{g}_1$ be two differential graded Lie algebras. Then the tautological linear embedding $\flat\colon \mathfrak{g}_0\to \mathfrak{g}_0[1]\hookrightarrow\cone(\mathrm{Id}_{\mathfrak{g}_0})$
induces a natural bijection between the set of differential graded Lie algebras morphisms
$\cone(\mathrm{Id}_{\mathfrak{g}_0})\xrightarrow{} \mathfrak{g}_1$ and the set of Cartan homotopies 
$\bi\colon\mathfrak{g}_0\xrightarrow{} \mathfrak{g}_1[-1]$.
Moreover, for any DG-Lie algebra morphism $\cone(\mathrm{Id}_{\mathfrak{g}_0})\xrightarrow{} \mathfrak{g}_1$,  the composition $\mathfrak{g}_0\hookrightarrow \cone(\mathrm{Id}_{\mathfrak{g}_0})\xrightarrow{}{\mathfrak{g}_1}$ is the boundary $\bl$ of the corresponding Cartan homotopy. 
\end{lemma}

\begin{proof}   
The composition
$\mathfrak{g}_0\xrightarrow{\flat} \mathfrak{g}_0[1]\hookrightarrow\cone(\mathrm{Id}_{\mathfrak{g}_0})$ 
 of $\flat$ with the tautological embedding of $ \mathfrak{g}_0[1]$ into $\cone(\mathrm{Id}_{\mathfrak{g}_0})$
induces a pullback map  between the space 
of morphisms of DG-vector spaces $\mathfrak{g}_0\oplus \mathfrak{g}_0[1]\to \mathfrak{g}_1$ and $\Hom^{-1}(\mathfrak{g}_0,\mathfrak{g}_1)$.
It is completely straightforward to prove that this map is bijective and that 
$\bl+\bi[1]\colon \cone(\mathrm{Id}_{\mathfrak{g}_0})\to \mathfrak{g}_1$ is a morphism of DG-Lie algebras if and only if 
$\bi\colon \mathfrak{g}_0\to \mathfrak{g}_1[-1]$ is a Cartan homotopy with boundary $\bl$. 
\end{proof}

\begin{definition}\label{def.cartancalculus}
A \emph{Cartan calculus} is a commutative square   of 
differential graded Lie algebras  of the form:
\begin{equation*}\label{equ.cartancalculus} 
\xymatrix{\mathfrak{g}_{00}\ar@{^{ (}->}[d]\ar[r]^-{h}&\mathfrak{g}_{10}\ar[d]^{\chi}\\
\cone(\mathrm{Id}_{\mathfrak{g}_{00}})\ar[r]^-{\bl+\bi[1]}&\mathfrak{g}_{11}}
\end{equation*}
\end{definition}

By Lemma~\ref{lem.universalCartanhomotopy}, giving a Cartan calculus is the same as giving two morphisms of DG-Lie algebras $\mathfrak{g}_{00}\xrightarrow{\,h\,}\mathfrak{g}_{10}\xrightarrow{\,\chi\,}\mathfrak{g}_{11}$ and a Cartan homotopy
$\bi\colon \mathfrak{g}_{00}\to \mathfrak{g}_{11}$ with boundary $\bl=\chi h$.  For later use we point out that the composition 
\[ \mathfrak{g}_{00}\xrightarrow{\;\bi\;} \mathfrak{g}_{11}[-1]\xrightarrow{\;\text{projection}\;}\Coker(\mathfrak{g}_{1\bullet})[-1]\;.\]
is a morphism of complexes.

\begin{proposition}\label{def.from.cartan}
A Cartan calculus $\mathfrak{g}_{00}\xrightarrow{\,h\,}\mathfrak{g}_{10}\xrightarrow{\,\chi\,}\mathfrak{g}_{11}$ with Cartan homotopy
$\bi\colon \mathfrak{g}_{00}\to \mathfrak{g}_{11}$ induces a distinguished morphism 
\[
\mathfrak{g}_{00}\to hofib(\mathfrak{g}_{1\bullet})
\]
in the homotopy category of DG-Lie algebras, presented by the span of DG-Lie algebras
\[
\xymatrix{ \mathfrak{g}_{00}&TW({\begin{smallmatrix}\mathfrak{g}_{00}\hookrightarrow\cone(\mathrm{Id}_{\mathfrak{g}_{00}})\end{smallmatrix}})\ar@{->>}[l]_-{\sim}\ar[rr]^-{(h,\bl+\bi[1])}&&TW(\mathfrak{g}_{1\bullet})}.
\]
The associated morphism of deformation functors $\Def_{\mathfrak{g}_{00}}\to \Def_{\mathfrak{g}_{1\bullet}}$ is induced by the morphism of Maurer-Cartan functors
\begin{align*}
\MC_{\mathfrak{g}_{00}}&\to \MC_{\mathfrak{g}_{1\bullet}}\\
x&\mapsto (h(x),e^{-\bi_x}).
\end{align*}
\end{proposition}
\begin{proof}
By functoriality of the Thom-Whitney homotopy fibers 
we get the
morphism of differential graded Lie algebras 
\[
TW(\mathfrak{g}_{00}\hookrightarrow\cone(\mathrm{Id}_{\mathfrak{g}_{00}}))\to TW(\mathfrak{g}_{1\bullet})\,.
\]
The conclusion  follows from  Lemma~\ref{lem.modellopiccolofibraomotopica}, Lemma~\ref{isomorphisms.def}, and Lemma~\ref{lem.sezionecartan}.
\end{proof}

\begin{remark}\label{rem.cartanincohomologia} In the setup of Proposition~\ref{def.from.cartan}, the induced morphism in cohomology $H^*(\mathfrak{g}_{00})\to H^*(hofib(\mathfrak{g}_{1\bullet}))$ admits a simple description whenever
$\mathfrak{g}_{10}\to \mathfrak{g}_{11}$ is injective: via the isomorphism $H^*(hofib(\mathfrak{g}_{1\bullet}))\cong H^*(TW(\mathfrak{g}_{1\bullet}))\xrightarrow{\sim} H^*(\Coker(\mathfrak{g}_{1\bullet})[-1])$ of 
Lemma~\ref{lemma.tw-quotient}, the composite  map 
$H^*(\mathfrak{g}_{00})\to H^*(\Coker(\mathfrak{g}_{1\bullet})[-1])$ is induced by the morphism of complexes $\mathfrak{g}_{00}\xrightarrow{\;\bi\;}\Coker(\mathfrak{g}_{1\bullet})[-1]$.
\end{remark}

\begin{remark}\label{rem.morphism-of-Cartan-calculi-to-span}
One has an obvious notion of morphisms of Cartan calculi: they are commutative diagrams of DG-Lie algebras of the form
\[
\xymatrix@!0{
& \mathfrak{g}_{001} \ar[rrr]^-{}\ar'[d][dd]
& & &\mathfrak{g}_{101} \ar[dd]^{}
\\
\mathfrak{g}_{000} \ar[ur]^{\varphi}\ar[rrr]^-{}\ar[dd]
& & &\mathfrak{g}_{100} \ar[ur]\ar[dd]
\\
& \text{\phantom{ii}\small{cone}}(\mathrm{Id}_{\mathfrak{g}_{001}}) \ar'[rr]^-{}[rrr]
& &  &\mathfrak{g}_{111}
\\
\text{\small{cone}}(\mathrm{Id}_{\mathfrak{g}_{000}})\ar[rrr]\ar[ur]
& & &\mathfrak{g}_{110} \ar[ur]
 \ar@{..}(14.5,-24.5);(14.5,-24.5)^{\text{\tiny{cone}}(\varphi)}
 \ar@{..}(30,-15.5);(30,-15.5)^{}
 \ar@{..}(21,-29.5);(21,-29.5)^{}
 }
\]
The above construction is manifestly functorial with respect to 
morphisms of Cartan calculi. In particular, given a morphism of Cartan calculi as above, we get a natural morphism 
\[
hofib(\mathfrak{g}_{00\bullet})\to hofib^{[2]}(\mathfrak{g}_{1\bullet\bullet})\]
in the homotopy category of DG-Lie algebras, presented by the span of DG-Lie algebras
\[ \xymatrix{
TW(\mathfrak{g}_{00\bullet})
&TW^{[2]}\left( {\begin{matrix}
\mathfrak{g}_{000}&\longrightarrow&\mathfrak{g}_{001}\\
\downarrow&& \downarrow\\
\cone(\mathrm{Id}_{\mathfrak{g}_{000}})&\to&\text{cone}(\mathrm{Id}_{\mathfrak{g}_{001}})
\end{matrix}}
\right)
\ar@{->>}[l]_-{\sim}\ar[r]&TW^{[2]}(\mathfrak{g}_{1\bullet\bullet})}.
\]

\end{remark}

\begin{proposition}\label{prop.double-fiber-in-cohomology}
Let $hofib(\mathfrak{g}_{00\bullet})\to hofib^{[2]}(\mathfrak{g}_{1\bullet\bullet})$ be the morphism in the homotopy category of DG-Lie algebras induced by a morphism of Cartan calculi 
\[
\xymatrix@!0{
& \mathfrak{g}_{001} \ar[rrr]^-{h_1}\ar'[d][dd]
& & &\mathfrak{g}_{101} \ar[dd]^{}
\\
\mathfrak{g}_{000} \ar[ur]^{\varphi}\ar[rrr]^-{h_0}\ar[dd]
& & &\mathfrak{g}_{100} \ar[ur]\ar[dd]
\\
& \quad\cone(\mathrm{Id}_{\mathfrak{g}_{001}}) \ar'[rr]^-{}[rrr]
& &  &\mathfrak{g}_{111}
\\
{\small{\cone}}(\mathrm{Id}_{\mathfrak{g}_{000}})\ar[rrr]\ar[ur]
& & &\mathfrak{g}_{110} \ar[ur]
 \ar@{..}(15.8,-24.2);(15.8,-24.2)^{\cone(\varphi)}
 \ar@{..}(30,-15.5);(30,-15.5)^{}
 \ar@{..}(21,-29.3);(21,-29.3)^{\tiny{\bl_0+\bi_0[1]}}
 }
\]
as in Remark \ref{rem.morphism-of-Cartan-calculi-to-span}. If the morphism $\mathfrak{g}_{000}\to \mathfrak{g}_{001}$ is the inclusion of a DG-Lie subalgebra and $\mathfrak{g}_{1\bullet\bullet}$ is a pullback of inclusions of DG-Lie algebras, then the induced morphism at the cohomology level
\[
H^*(hofib(\mathfrak{g}_{00\bullet}))\to H^*(hofib^{[2]}(\mathfrak{g}_{1\bullet\bullet}))
\]
is naturally identified with the morphism
\begin{align*}
H^{*-1}(\mathfrak{g}_{001}/\mathfrak{g}_{000})&\to H^{*-2}(\mathfrak{g}_{111}/(\mathfrak{g}_{110}+\mathfrak{g}_{101}))\\
[x]&\mapsto [\bi_{1 \tilde{x}} \mod \mathfrak{g}_{110}+\mathfrak{g}_{101}]
\end{align*}
where $\tilde{x}$ is any representative of $[x]$ in $\mathfrak{g}_{001}$.
\end{proposition}

\begin{proof}
Since $\mathfrak{g}_{1\bullet0}$ and $\mathfrak{g}_{1\bullet1}$ are inclusions of sub-DG-Lie algebras, we have $h_0=\bl_0$ and $h_1=\bl_1$. Let $[x]$ be a cohomology class in $H^{*-1}(\mathfrak{g}_{001}/\mathfrak{g}_{000})$. By Remark \ref{rem.representative}, a representative for $[x]$ in $TW(\mathfrak{g}_{00\bullet})$ is $(d\tilde{x},d(t\, \tilde{x}))$, where $\tilde{x}$ is any representative of $[x]$ in $\mathfrak{g}_{001}$. By Remark \ref{rem.g-to-TW}, a representative for $[x]$ in 
$TW^{[2]}\left( {\begin{smallmatrix}
\mathfrak{g}_{000}&\longrightarrow&\mathfrak{g}_{001}\\
\downarrow&& \downarrow\\
\text{\tiny{cone}}(\mathrm{Id}_{\mathfrak{g}_{000}})&\to&\text{\tiny{cone}}(\mathrm{Id}_{\mathfrak{g}_{001}})
\end{smallmatrix}}
\right)$ is therefore given by 
\[(d\tilde{x}, t_0\, d\tilde{x}+dt_0 \flat d\tilde{x}, d(t_1\, \tilde{x}), t_0\, d(t_1\, \tilde{x})+dt_0 \flat d(t_1\, \tilde{x})),\] 
where $(t_0,t_1)$ are coordinates on $\Delta^1\times\Delta^1$. This is mapped to the element
\[
(d\bl_{0\tilde{x}}, t_0\, d\bl_{0\tilde{x}}+dt_0 \bi_{0 d\tilde{x}}, dt_1\, \bl_{1\tilde{x}}+t_1d\bl_{1\tilde{x}}, t_0\, dt_1\,\bl_{1 \tilde{x}}+t_1\,d\bl_{1 \tilde{x}}+dt_0 dt_1\bi_{1 \tilde{x}}+t_1dt_0 \bi_{1 d\tilde{x}})
\]
in $TW^{[2]}(\mathfrak{g}_{1\bullet\bullet})$. Finally, by Corollary \ref{cor.cohomology-minus-2}, this element is mapped to the element $\bi_{1 \tilde{x}} \mod \mathfrak{g}_{110}+\mathfrak{g}_{101}$ in $\mathfrak{g}_{111}/(\mathfrak{g}_{110}+\mathfrak{g}_{101})[-2]$.

\end{proof}

\bigskip
\section{The Jacobian square and the Jacobian DG-Lie algebra}\label{section.end.aff.etc}

The goal of this section is to define, for every cofibration  $F\hookrightarrow V$ in 
$\mathbf{DG}$, i.e., for every injective morphisms of DG-vector spaces, a commutative square  $\mathcal{J}(F\hookrightarrow V)\in Fun(\Delta^1\times \Delta^1,\mathbf{DGLA})$ in such a way that if $F\hookrightarrow V$ is weakly equivalent to $\tilde{F}\hookrightarrow \tilde{V}$ in 
$Fun(\Delta^1,\mathbf{DG})$, then $\mathcal{J}(F\hookrightarrow V)$ is weakly equivalent to $\mathcal{J}(\tilde{F}\hookrightarrow \tilde{V})$ in $Fun(\Delta^1\times \Delta^1,\mathbf{DGLA})$. Once this is achieved, 
one can associate to any morphism $F\to V$ in $Fun(\Delta^1,\mathbf{DG})$ a well defined object in the homotopy category of $Fun(\Delta^1\times \Delta^1,\mathbf{DGLA})$ by taking a factorization 
\[F\hookrightarrow V'\stackrel{\sim}{\twoheadrightarrow} V\]
and considering the homotopy class of the square $\mathcal{J}(F\hookrightarrow V')$. 
Then we define the DG-Lie algebra $\mathfrak{J}(V,F)$ as the double homotopy fibre of $\mathcal{J}(F\hookrightarrow V)$,  which controls in a natural way the local structure of the intermediate Jacobian of the  cofibration  $F\hookrightarrow V$.

For notational simplicity we shall identify cofibrations in $\mathbf{DG}$ with inclusions of DG-subspaces and will denote by $(V,F)$ a pair where $F$ is a DG-vector subspace of $V$. 
By the term \emph{pair of DG-vector spaces} we will always mean a pair consisting of a DG-vector space and of one of its DG-subspaces.\footnote{Maybe \emph{DG-flag} would be a more evocative name.} Thus 
a morphism of 
pairs $\xi\colon (V,F)\to (\tilde{V},\tilde{F})$ is a morphism of DG-vector spaces $\xi\colon V\to \tilde{V}$ such that 
$\xi(F)\subseteq \tilde{F}$; a morphism $\xi\colon (V,F)\to (\tilde{V},\tilde{F})$ is a weak equivalence if and only if both $\xi\colon V\to \tilde{V}$ and $\xi\colon F\to \tilde{F}$ are quasi-isomorphisms.  Also, in what follows we will write $\mathcal{J}(V,F)$ for $\mathcal{J}(F\hookrightarrow V)$.

\begin{notation}
For a DG-vector space $V$ and  $F_1,\ldots,F_n$ DG-subspaces of $V$ we set
 \[  \End^*(V)=\Hom^*_{\K}(V,V);\qquad \End^*(V;F_1,\ldots,F_n)=\{\xi\in \End^*(V)\mid \xi(F_i)\subseteq F_i\quad \forall i\}.\]
These are DG-Lie algebras with the usual commutator bracket of linear endomorphisms of graded vector spaces. We also set
\[ \Aff(V)=\{ \tilde{\xi}\in \End^*(V\oplus\K)\mid \tilde{\xi}(V\oplus\K)\subseteq V\}\]
and
\[ \Aff(V;F_1,\ldots,F_n)=\{ \tilde{\xi}\in \End^*(V\oplus\K)\mid \tilde{\xi}(V\oplus \K)\subseteq V,\; \tilde{\xi}(F_i\oplus\K)\subseteq F_i\quad \forall i\},\]
and call these the DG-Lie algebra of affine
endomorphisms of $V$ and of $(V;F_1,\ldots,F_n)$, respectively.
\end{notation}

\begin{remark}\label{rem.aff}
The DG-Lie algebra of affine
endomorphisms $\Aff(V)$ is canonically isomorphic to the differential graded vector space $\End^*(V)\oplus V$,  endowed with the Lie bracket 
\[ [(\xi,w),(\eta,u)]=[\xi,\eta]+(\xi(u)-(-1)^{\bar{\eta}\;\bar{w}}\eta(w))\;.\]
Namely, the pair $(\xi,w)\in \End^*(V)\oplus V$ corresponds to the endomorphism of $V\oplus\K$ represented by the matrix
\[
\left(
\begin{matrix} \xi & w\\
0 & 0
\end{matrix}
\right)
\]
Similarly, one sees that $\Aff(V;F_1,\ldots,F_n)$ is  canonically isomorphic to the sub-DG-Lie algebra $\End^*(V;F_1,\ldots,F_n)\oplus F_n$ of $\End^*(V)\oplus V$.
\end{remark}

\begin{remark}\label{rem.sezionidistorte} 
Notice that every degree zero closed element 
$v\in Z^0(V)$ determines a DG-Lie algebra section of the projection $\Aff(V)\to \End^*(V)$, given by  
\[\sigma_v\colon\End^*(V)\to\Aff(V),\qquad
\sigma_v(\xi)=(\xi,-\xi(v)).\]
This is the canonical identification of $\End^*(V)$ with the stabilizer of $v$ under the action of $\Aff(V)$ on $V$. In particular $\sigma_0$ is the canonical embedding of $\End^*(V)$ into $\Aff(V)$ given by $\xi\mapsto (\xi,0)$.
\end{remark}

\begin{definition}\label{def.jacobian-dgla}
For a pair of DG-vector spaces $(V,F)$, its Jacobian diagram is the  
commutative square  of differential graded Lie algebras 
\begin{equation}\label{equ.jacobiandiagram} 
\mathcal{J}(V,F)_{\bullet\bullet}\colon\qquad
\xymatrix{\End^*(V;F)\ar[d]_{\sigma_0}\ar[r]& \End^*(V)\ar[d]^{\sigma_0}\\
\Aff^*(V;F)\ar[r]& \Aff^*(V)}
\end{equation}
where
every morphism in the diagram is the natural inclusion.
The \emph{Jacobian DG-Lie algebra} $\mathfrak{J}(V,F)$ of the pair $(V,F)$ is the double
homotopy fiber of its Jacobian square:
\[
\mathfrak{J}(V,F):=hofib^{[2]}(\mathcal{J}(V,F)_{\bullet\bullet})
\]
\end{definition}

We now want to show that the Jacobian DG-Lie algebra $\mathfrak{J}(V,F)$ only depends on the homotopy equivalence class of the pair $(V,F)$. To this end it is useful to consider in $Fun(\Delta^1,\mathbf{DG})$ the direct Reedy model structure, i.e.,  one where the map $0\to 1$ has positive degree: in this structure a map  of DG-pairs $f\colon (V,F)\to (\tilde{V},\tilde{F})$ is a fibration (resp.: weak equivalence) 
if and only if  both $f\colon V\to \tilde{V}$ and $f\colon F\to \tilde{F}$ are surjective 
(resp.: quasi-isomorphisms).  

\begin{lemma}\label{for_Ken_Brown} Let $(V,F)\stackrel{\sim}{\twoheadrightarrow}(\tilde{V},\tilde{F})$ be a trivial fibration in the direct Reedy structure on $Fun(\Delta^1,\mathbf{DG})$.
Then $\mathcal{J}(V,F)_{\bullet\bullet}$ and $\mathcal{J}(\tilde{V},\tilde{F})_{\bullet\bullet}$ are homotopy equivalent.
\end{lemma}

\begin{proof}
If $\alpha\colon (V,F)\to (\tilde{V},\tilde{F})$ is a trivial fibration, then in particular both $\alpha\colon V\to \tilde{V}$ and $\alpha\colon F\to \tilde{F}$ are surjective quasi-isomorphisms. Then, denoting by $I$ and $I_F$ the kernels of $\alpha\colon V\to \tilde{V}$ and $\alpha\vert_F\colon F\to \tilde{F}$, respectively, we see that both $I$ and $I_F$ are acyclic.  
This implies that both 
the natural inclusion $\imath\colon \End(V;F,I)\to \End^*(V;F)$ and the natural projection 
$\pi\colon \End^*(V;F,I)\to \End^*(\tilde{V};\tilde{F})$ are quasi-isomorphisms.
Indeed
both
the complexes
\[ G_{\alpha}=\{f\in \Hom^*_{\K}(I,V/I)\mid f(I_F)\subseteq F/I_F\},\]
\[ R_{\alpha}=\{f\in \End^*(V;F)\mid f(V)\subseteq I\},\]
are acyclic,
since
$G_{\alpha}$ is the kernel of the surjective morphism of acyclic complexes
\[ \Hom^*_{\K}(I,V/I)\to \Hom^*_{\K}(I_F,V/(F+I)),\]
while $R_{\alpha}$ is the kernel of the surjective morphism of acyclic complexes
\[ \Hom^*_{\K}(V,I)\to \Hom^*_{\K}(F,I/I_F).\]
To conclude, it is sufficient to consider the 
short exact sequences of DG-vector spaces
\[ 0\to \End^*(V;F,I)\xrightarrow{\imath} \End^*(V;F)\to G_{\alpha}\to 0,\]
\[ 0\to R_{\alpha}\to \End^*(V;F,I)\xrightarrow{\pi} \End^*(\tilde{V};\tilde{F})\to 0.\]
In particular, by taking $F=\tilde{F}=0$, we see that  
both 
the inclusion $\imath\colon \End^*(V;I)\to \End^*(V)$ and the projection 
$\pi\colon \End^*(V;I)\to \End^*(\tilde{V})$ are quasi-isomorphisms.
It is now clear that the dashed arrows in the two diagrams below give weak equivalences of squares of differential graded Lie algebras

\[
\scalebox{.90}{
\xymatrix@1{
   &  \End^*(V;F,I)\ar[rr] \ar@{-->}[ddd]^(0.6){\imath} \ar[dl]  & & 
   \End^*(V;I) \ar@{-->}[ddd]^(0.6){\imath} \ar[dl] \\
   \End^*(V;F,I)\oplus F\ar[rr] \ar@{-->}[ddd]^(0.4){\imath\oplus\Id_F}& & \End^*(V;I)\oplus V \ar@{-->}[ddd]^(0.4){\imath\oplus\Id_V}& \\
   &  &  &  & \\
   & \End^*(V;F)\ar[rr] \ar[dl] & & \End^*(V) \ar[dl] \\
   \End^*(V;F)\oplus F \ar[rr] & & \End^*(V)\oplus V & }}\]

\[\scalebox{.90}{\xymatrix@1{
   &  \End^*(V;F,I)\ar[rr] \ar@{-->}[ddd]^(0.6){\pi} \ar[dl]  & & 
   \End^*(V;I) \ar@{-->}[ddd]^(0.6){\pi} \ar[dl] \\
   \End^*(V;F,I)\oplus F\ar[rr] \ar@{-->}[ddd]^(0.4){\pi\oplus \alpha}& & \End^*(V;I)\oplus V \ar@{-->}[ddd]^(0.4){\pi\oplus \alpha}& \\
   &  &  &  & \\
   & \End^*(\tilde{V};\tilde{F})\ar[rr] \ar[dl] & & \End^*(\tilde{V}) \ar[dl] \\
   \End^*(\tilde{V};\tilde{F})\oplus \tilde{F} \ar[rr] & & \End^*(\tilde{V})\oplus \tilde{V} & }}\]

\end{proof}

\begin{corollary}\label{ken-brown} If $(V,F)$ and $(\tilde{V},\tilde{F})$ are homotopy equivalent DG-pairs, then the DG-Lie algebras 
$\mathfrak{J}(V,F)$ and $\mathfrak{J}(\tilde{V},\tilde{F})$ are homotopy equivalent.
\end{corollary}
\begin{proof}
Assume $(V,F)$ and $(\tilde{V},\tilde{F})$ are homotopy equivalent. 
Then, by Lemma \ref{for_Ken_Brown} and by the Ken Brown's lemma \cite[1.1.12]{Hov99}, we have that  $\mathcal{J}(V,F)_{\bullet\bullet}$ and $\mathcal{J}(\tilde{V},\tilde{F})_{\bullet\bullet}$ are homotopy equivalent commutative squares of DG-Lie algebras, and so by Proposition~\ref{homotopy-invariance-of-double-fiber}, the DG-Lie algebras $\mathfrak{J}(V,F)$ and $\mathfrak{J}(\tilde{V},\tilde{F})$ are homotopy equivalent.
\end{proof}

\begin{lemma}\label{lemma.J-easy} Let $(V,F)$ be a DG-pair. Then the DG-Lie algebra $\mathfrak{J}(V,F)$ is homotopy abelian. More precisely, $\mathfrak{J}(V,F)$ is homotopy equivalent to the complex $V/F[-2]$ equipped with trivial bracket.
\end{lemma}
\begin{proof}
Via the isomorphism $\Aff(V;F)\cong \End^*(V;F)\oplus F$ of Remark \ref{rem.aff} it is immediate to see that the morphism $\End^*(V;F)\to \Aff(V,F)$ in the commutative square of DG-Lie algebras $\mathcal{J}(V,F)_{\bullet\bullet}$ is injective in cohomology. Therefore, $\mathfrak{J}(V,F)$ is homotopy abelian by Corollary~\ref{cor.abelinitafibraquadrato}. As such, $\mathfrak{J}(V,F)$ is weakly equivalent to its cohomology as a DG-vector space (endowed with the trivial bracket) and so, by Lemma \ref{lem.hypercohomology}, 
\[
\mathfrak{J}(V,F)\simeq \mathrm{Tot}(\mathcal{J}(V,F)_{\bullet\bullet})\simeq V/F[-2],
\]
where the quasi-isomorphism $\Aff(V)/(\Aff(V;F)+\End(V))\xrightarrow{\sim}V/F$ is given by $[(\xi,v)]\mapsto [v \mod F]$.
\end{proof}

The above lemma implies in particular that he DG-Lie algebra $TW^{[2]}(\mathcal{J}(V,F)_{\bullet\bullet})$ is homotopy equivalent to the abelian DG-Lie algebra $V/F[-2]$. In particular the deformation functor $\Def_{\mathfrak{J}(V,F)}$ is isomorphic to the smooth functor
\[ \Def_{V/F[-2]}\colon \Art\to \Set,\qquad A\mapsto \frac{V^{-1}}{F^{-1}}\otimes\mathfrak{m}_A\,.\]
It may be argued that germs of smooth moduli spaces and smooth deformation functors are uninteresting; on the contrary, maps between them may be extremely interesting. As we shall see later, the formal moduli space controlled by the double homotopy fiber of the Jacobian square is the target of some nontrivial maps called called \emph{formal Abel-Jacobi maps} and then every explicit DG-Lie representative, such as  
$V/F[-2]$ and $TW^{[2]}(\mathcal{J}(V,F)_{\bullet\bullet})$, gives an explicit description of these maps. 
Since there is no royal road to geometry, it is not surprising that simple representatives such as $V/F[-2]$ provide complicated descriptions of formal Abel-Jacobi maps, while the Thom-Whitney representative will provide a
simple description. A  nontrivial explicit $L_\infty$ quasi-isomorphism from 
$TW^{[2]}(\mathcal{J}(V,F)_{\bullet\bullet})$ to $V/F[-2]$ will be 
described in Appendix~\ref{app.iterati} by means of iterated integrals.
\medskip

We conclude this section by introducing two additional DG-Lie algebras which will be relevant for what follows.

\begin{definition} Let $(V,F)$ be a DG-pair. The \emph{Grassmannian} DG-Lie algebra $\mathfrak{Grass}(V;F)$ is defined as
\[
\mathfrak{Grass}(V;F):=TW(\End^*(V;F)\hookrightarrow\End^*(V)).
\]
The \emph{tautological bundle} DG-Lie algebra $\mathfrak{Q}(V;F)$ is defined as
\[
\mathfrak{Q}(V;F):=TW(\Aff(V;F)\hookrightarrow\Aff(V)).
\]
\end{definition}
\begin{remark}\label{rem.homotopy-invariance-grass}
The DG-Lie algebras $\mathfrak{Grass}(V;F)$ and $\mathfrak{Q}(V;F)$ are models for $hofib(\End^*(V;F)\hookrightarrow\End^*(V))$ and $hofib(\Aff(V;F)\hookrightarrow\Aff(V))$ respectively. Notice that the same argument as in the proof of Lemma \ref{for_Ken_Brown} and Corollary \ref{for_Ken_Brown} show that the homotopy type of $\mathfrak{Grass}(V;F)$ and $\mathfrak{Q}(V;F)$ only depends on the homotopy type of the DG-pair $(V,F)$.
Moreover, the Jacobian square induces an injective morphism of DG-Lie algebras $\mathfrak{Grass}(V;F)\hookrightarrow \mathfrak{Q}(V;F)$ giving a weak equivalence
\[
\mathfrak{J}(V,F)\simeq TW(\mathfrak{Grass}(V;F)\hookrightarrow\mathfrak{Q}(V;F)).
\]
Finally, the natural projections $\Aff(V)\to \End^*(V)$ and $\Aff(V;F)\to \End^*(V;F)$ induce a surjective DG-Lie algebra morphism $ \mathfrak{Q}(V;F)\twoheadrightarrow \mathfrak{Grass}(V;F)$ of which the above morphism  $\mathfrak{Grass}(V;F)\hookrightarrow \mathfrak{Q}(V;F)$ is a section.
\end{remark}

\begin{proposition}\label{prop.injectiveincohomologygrassabelian} 
If the inclusion $F\hookrightarrow V$ is injective in cohomology, then both the DG-Lie algebras 
$\mathfrak{Grass}(V;F)$ and $\mathfrak{Q}(V;F)$ are homotopy abelian. 
\end{proposition}
\begin{proof} By the K\"{u}nneth formula, if $F\hookrightarrow V$ is injective in cohomology then so are both $\End^*(V;F)\hookrightarrow\End^*(V)$ and $\Aff(V;F)\hookrightarrow\Aff(V)$. The conclusion is then immediate by Corollary~\ref{fiber-quasi-abelian}.
\end{proof}

\bigskip
\section{The (formal infinitesimal) homotopy Grassmannian}
\label{sec.grassmann}

We will now start investigating the infinitesimal deformation functors associated with the DG-Lie algebra $\mathfrak{Grass}(V;F)$. As the notation suggests, this will turn out to be the deformation functor describing the infinitesimal deformations of $F$ as a subcomplex of $V$: an infinitesimal and homotopy invariant version of the usual Grassmannian of deformations of a subspace of a given vector space. According to Lemma~\ref{lem.modellopiccolofibraomotopica}, to give a geometrical interpretation of the deformation functor $\Def_{\mathfrak{Grass}(V;F)}$ is the same as giving a geometric description of the functor
$\Def_{\End^*(V;F)\hookrightarrow\End^*(V)}$. It will be convenient to introduce the group-valued functors of infinitesimal automorphisms of a differential graded vector space, as follows.
\begin{definition}
Let $V$ be a differential graded vector space. The functors of Artin rings  
\[
\Aut_V,\Aut_V^0\colon\Art\to \mathbf{Grp}\] are defined as:
\[ \begin{split}\Aut_V(A)&=\{\text{automorphisms of the graded $A$-module $V\otimes A$ lifting the identity on $V$}\};\\
\Aut_V^0(A)&=\left\{\begin{array}{c}
\text{$A$-linear automorphisms of the complex $V\otimes A$, lifting the }\\
\text{identity  on $V$  and  inducing the identity in $H^*(V\otimes A)$} \end{array}\right\}.\end{split}\]
\end{definition}
\begin{remark}
Since the maximal ideal $\mathfrak{m}_A$ of a local Artin algebra $A$ is nilpotent, we can write
\[ \Aut_V(A)=\{ \Id_{V\otimes A}-g\,,\, \text{with $g\colon V\otimes A\to V\otimes\mathfrak{m}_A$ a degree $0$ $A$-linear map}\}\]
where $g$ is any $A$-linear map $g\colon V\otimes A\to V\otimes\mathfrak{m}_A$ of degree $0$.
Notice that for such a $g$ we can write $\Id_{V\otimes A}-g=e^a$ with
\[a=\log(\Id_{V\otimes A}-g)=-\sum_{n=1}^{\infty}\frac{g^n}{n}\] 
in $\End^0(V)\otimes\mathfrak{m}_A$.
\end{remark}

\begin{lemma}\label{serve.anche.questo}
Let $V$ be a differential graded vector space and let $\xi\in \End^0(V)\otimes\mathfrak{m}_A$, then:

\begin{enumerate} 

\item $e^\xi(F\otimes A)=F\otimes A$ if and only if $\xi\in 
\End^0(V;F)\otimes\mathfrak{m}_A$;

\item $e^\xi$ is a morphism of complexes if and only if $d\xi=0$;

\item $e^\xi\in \Aut^0_V(A)$  if and only if $\xi=d\zeta$ for some $\zeta\in \End^{-1}(V)\otimes\mathfrak{m}_A$.
\end{enumerate}
\end{lemma}

\begin{proof} Only the last item is nontrivial. We have $e^\xi\in \Aut^0_V(A)$ if and only if $e^\xi-\Id$.
is the zero morphism in cohomology. By writing 
\[ e^\xi-\Id=\frac{e^\xi-\Id}{\xi} \xi\]
and noticing that $\frac{e^\xi-\Id}{\xi}$ is an isomorphism, we see that
$e^\xi\in \Aut^0_V(A)$ if and only if $\xi\colon V\otimes A\to V\otimes A$ is zero in cohomology and the conclusion follows by K\"{u}nneth formulas.
\end{proof}

\begin{definition}
For any DG-pair $(F,V)$ and for any local Artin algebra $A$, we denote by $S_{V,F}(A)$ the set
\[ S_{V,F}(A)=\{F_A\mid F_A\subseteq V\otimes A\; \text{subcomplex of flat $A$-modules such that}\; 
F_A\otimes_A\K=F\}.\]
\end{definition}
\begin{lemma}For every morphism $A\to B$ in $\mathbf{Art}$, the map $F_A\mapsto F_A\otimes_A B$ defines a morphism $S_{V,F}(A)\to S_{V,F}(B)$ making  $S_{V,F}\colon \Art\to \Set$  a functor of Artin rings.
\end{lemma}
\begin{proof}
The functoriality of $S_{V,F}$ is an easy consequence of the flatness criteria for modules over a local Artin rings. Recall that an $A$-module is flat if and only if it is free. Applying the functor $-\otimes_A\K$ to the exact sequences 
\begin{equation}\label{equ.tautologicograssmann}
0\to F_A^i\to V^i\otimes A\to \frac{V^i\otimes A}{F_A^i}\to 0,\qquad i\in \Z,
\end{equation}
we get $\operatorname{Tor}_1^A(V^i\otimes A/F_A^i,\K)=0$ and then also $V^i\otimes A/F_A^i$ is a flat $A$-module. Thus, for every morphism $A\to B$ in $\Art$ we have 
$\operatorname{Tor}_j^A(V^i\otimes A/F_A,B)=0$ for every $j>0$ and then  
$F_A\otimes_A B\subseteq V\otimes B$ is a subcomplex of flat $B$-modules.
\end{proof}

\begin{definition}
For any DG-pair $(F,V)$ we denote by $Grass_{V,F}\colon \Art\to \Set$ the quotient functor
\[  Grass_{V,F}=S_{V,F}/\Aut^0_V,
\]
where the $\Aut^0_V$-action on $S_{V,F}$ is the
natural left action 
\[ \Aut^0_V\times S_{V,F}\to S_{V,F},\qquad (e^a,F_A)\mapsto e^a(F_A).\]
\end{definition}
\begin{remark}
When $V=V^0$ is a $\mathbb{K}$ vector space (seen as a differential graded vector space concentrated in degree zero) and $F=F^0$ is a linear subspace of $V$, then the functor  $Grass_{V,F}$ is the usual Grassmann functor of 
deformations of $F$ inside $V$. More generally if the differential of $V$ is trivial, then 
$Grass_{V,F}$ is the product of the Grassmann functors $Grass_{V^i,F^i}$.
\end{remark}

\begin{lemma}\label{first.phi}
Let $(V,F)$ be a DG-pair. The map $\phi_A\colon e^\xi\mapsto e^{-\xi}(F\otimes A)$, with $A$ a local Artin algebra and $e^a\in \exp(\End^0(V)\otimes \mathfrak{m}_A)$, defines a natural transformation of functors
\[
\phi\colon \MC_{\End^*(V;F)\hookrightarrow\End^*(V)}\to S_{V,F}.
\]
\end{lemma}
\begin{proof}
By Remark \ref{rem.injective}, we need to show that, if $e^\xi$ is such that $e^\xi*0\in \End^1(V;F)\otimes \mathfrak{m}_A$, then $e^{-\xi}(F\otimes A)$ is a subcomplex of $V\otimes A$. This is immediate since $e^{-\xi}(F\otimes A)$ is a subcomplex of $V\otimes A$ if and only if $de^{-\xi}(F\otimes A)\subseteq e^{-\xi}(F\otimes A)$ and so if and only if $e^\xi de^{-\xi}(F\otimes A)\subseteq F\otimes A$. This last condition can be rephrased by saying that one wants $e^\xi de^{-\xi}-d$ to be an element in $\End^1(V;F)\otimes \mathfrak{m}_A$, and then one concludes by noticing that $e^\xi de^{-\xi}-d=e^\xi*0$.
\end{proof}

\begin{lemma}\label{second.phi}
Let $(V,F)$ be a DG-pair. The natural transformation $\phi$ of Lemma \ref{first.phi} induces an isomorphism
\[
\phi_A\colon  \frac{\MC_{\End^*(V;F)\hookrightarrow\End^*(V)}(A)}{\exp(\End^0(V;F)\otimes \mathfrak{m}_A)}\xrightarrow{\;\sim\;} S_{V,F}(A)\,,
\]
for any local Artin algebra $A$.
\end{lemma}
\begin{proof} Given $F_A\in S_{V,F}(A)$, every basis of $F$ lifts to a basis of $F_A$ and then there exists 
an element $\xi\in \End^{0}(V)\otimes \mathfrak{m}_A$ such that $e^{-\xi}(F\otimes A)=F_A$. The argument in the proof of Lemma \ref{first.phi} shows that $e^\xi\in \MC_{\End^*(V;F)\hookrightarrow\End^*(V)}(A)$. Therefore $\phi_A$ is surjective and to conclude we need to show that $\phi_A(e^{\xi_1})=\phi_A(e^{\xi_2})$ if and only if $e^{\xi_2}=e^\eta e^{\xi_1}$ for some $\eta\in \End^0(V;F)\otimes \mathfrak{m}_A$.
This is immediate, since  we have 
$e^{-\xi_1}(F\otimes A)=e^{-\xi_2}(F\otimes A)$ if and only if $e^{\xi_2}e^{-\xi_1}(F\otimes A)\subseteq F\otimes A$, i.e., if and only if $e^{\xi_2}e^{-\xi_1}=e^\eta$ with $\eta\in \End^0(V;F)\otimes \mathfrak{m}_A$.
\end{proof}

\begin{proposition}\label{prop.grassmmannfunctor}
Let $(V,F)$ be a DG-pair. The natural transformation $\phi$ of Lemma \ref{first.phi}  induces a natural isomorphism of set-valued functors of local Artin algebras $\phi\colon \Def_{\mathfrak{Grass}(V;F)}\xrightarrow{\sim} Grass_{V,F}$. In particular, $Grass_{V,F}$ is an infinitesimal deformation functor.
\end{proposition}
\begin{proof} From Lemma \ref{second.phi} it follows that $\phi$ induces a natural isomorphism
\[
\phi_A\colon \frac{\MC_{\End^*(V;F)\hookrightarrow\End^*(V)}(A)}{\;\exp(\End^0(V;F)\otimes \mathfrak{m}_A)\times \Aut^0_V(A)\;}\xrightarrow{\;\sim\;}  
Grass_{V,F}(A),
\]
for any local Artin algebra $A$. Now recall from Lemma \ref{serve.anche.questo} that $\Aut^0_V(A)\cong \exp(d\End^{-1}(V)\otimes \mathfrak{m}_A)$, and use Lemma \ref{lem.modellopiccolofibraomotopica} to conclude.
\end{proof}

\begin{remark}
 Unless the inclusion $F\hookrightarrow V$ is injective in cohomology, the deformation functor $Grass_{V,F}$ is generally obstructed. Consider for instance the case where $V=V^0\oplus V^1$, $\dim V^0=\dim V^1=2$, the differential $d\colon V^0\to V^1$ has rank $1$ and $F=\ker d\oplus d(V^0)$. 
Then the functor $Grass_{V,F}$ is pro-represented by two lines meeting transversally at a point: the first line parametrizes deformations of $F$ leaving $F^0=\ker d$ fixed and the second line parametrizes deformations of $F$ leaving $F^1=d(V^0)$ fixed. Alternatively, one can  easily prove that the primary obstruction of the functor $\Def_{\mathfrak{Grass}(V;F)}$ is not trivial.
\end{remark}

\bigskip
\section{The tautological bundle of the (infinitesimal) homotopy Grassmannian}

We now turn our attention to the deformation functor $Q_{V,F}$ associated with the DG-Lie algebra $\mathfrak{Q}(V;F)$. We will show that this deformation functor can be naturally interpreted as the tautological bundle over the infinitesimal homotopy Grassmannian $Grass_{V,F}$ and that the two morphism of DG-Lie algebras $\mathfrak{Grass}(V,F)\hookrightarrow \mathfrak{Q}(V;F)\twoheadrightarrow \mathfrak{Grass}(V,F)$ considered in Section \ref{section.end.aff.etc} induce the zero section and the projection for the tautological bundle, respectively.

\begin{definition}
For any DG-pair $(F,V)$ and for any local Artin algebra $A$, we denote by $Z_{V,F}$ and $E_{V,F}(A)$ the sets
\[ Z_{V,F}(A)=\left\{(F_A,v)\;\middle|\; F_A\in S_{V,F}(A),\quad v\in Z^0\left(\frac{V\otimes\mathfrak{m}_A}{F_A\otimes_A\mathfrak{m}_A}\right)\right\}\]
and
\[ E_{V,F}(A)=\left\{(F_A,[v])\;\middle|\; F_A\in S_{V,F}(A),\quad [v]\in H^0\left(\frac{V\otimes\mathfrak{m}_A}{F_A\otimes_A\mathfrak{m}_A}\right)\right\}.\]
\end{definition}
\begin{lemma}Every morphism $A\to B$ in $\mathbf{Art}$ induces a morphism \[\frac{V\otimes\mathfrak{m}_A}{F_A\otimes_A\mathfrak{m}_A}\to 
\frac{V\otimes\mathfrak{m}_B}{(F_A\otimes_A B)\otimes_B\mathfrak{m}_B}\;.\] making $Z_{V,F}, E_{V,F}\colon \Art\to \Set$ be functors. The obvious projection $(F_A,v)\mapsto (F_A,[v])$ is a natural transformation $Z_{V,F}\to E_{V,F}$.
\end{lemma}
\begin{proof}
This is immediate, since a morphism $A\to B$ induces a morphisms $F_A\to F_A\otimes_A B$. 
\end{proof}

The action of $\Aut^0_V$ on $S_{V,F}$ lifts naturally to an action on $E_{V,F}$, since every 
$e^\xi\in \Aut_V^0(A)$ gives a morphism of cochain complexes
\[ e^\xi\colon \frac{V\otimes\mathfrak{m}_A}{F_A\otimes_A\mathfrak{m}_A}\to
\frac{V\otimes\mathfrak{m}_A}{e^\xi(F_A)\otimes_A\mathfrak{m}_A}\;.\]
\begin{definition}
For any DG-pair $(F,V)$ we denote by $Q_{V,F}\colon \Art\to \Set$ the quotient functor
\[  Q_{V,F}=E_{V,F}/\Aut^0_V,
\]
where the $\Aut^0_V$-action on $E_{V,F}$ is the
natural left action 
\[ \Aut^0_V\times E_{V,F}\to E_{V,F},\qquad (e^\xi,(F_A,[v]))\mapsto (e^\xi(F_A),e^\xi[v]).\]
\end{definition}

\begin{lemma}\label{first.psi}
Let $(V,F)$ be a DG-pair. Then we have a natural transformation 
\[ \psi\colon \MC_{\Aff(V;F)\hookrightarrow\Aff(V)}\to Z_{V,F},\]
given by
\[ \psi_A\colon (\xi,w)\to\left(e^{-\xi}(F\otimes A),\frac{e^{-\xi}-1}{\xi}(w)\mod e^{-\xi}(F\otimes A)\otimes_A \mathfrak{m}_A\right)\]
for any local Artin algebra $A$, where we have used the DG-vector space isomorphism $\Aff^*(V)=\End^*(V)\oplus V$ to write an element in $\Aff(V)^0\otimes\mathfrak{m}_A$ as a pair $(\xi,v)$, where $\xi\in\End^0(V)\otimes\mathfrak{m}_A$ and $v\in V^0\otimes\mathfrak{m}_A$.
\end{lemma}
\begin{proof}
As a DG-Lie algebra, $\Aff(V)$ is a sub-DG-Lie algebra of $\End^*(V\oplus \K)$ via the embedding 
\[
(\xi,w)\mapsto \left(\begin{matrix}
\xi&w\\0&0
\end{matrix}\right)
\]
so we can compute the gauge action of 
$\exp(\Aff^0(V)\otimes\mathfrak{m}_A)=\exp(\End^0(V)\otimes\mathfrak{m}_A)\ltimes V^0\otimes\mathfrak{m}_A$ by working in $\End^*(V\oplus \K)$, where this gauge action takes a simple form, expressed in terms of the differential graded associative algebra structure on $\End^*(V\oplus \K)$. Namely, in $\End^*(V\oplus \K)\otimes\mathfrak{m}_A$ we have
\begin{align*}
\exp\left(\begin{matrix}\xi & w\\0&0\end{matrix}\right)d \exp\left(\begin{matrix}-\xi & -w\\0&0\end{matrix}\right)-d&=
\left(
\begin{matrix}e^{\xi}&\frac{e^{\xi}-1}{\xi}(w)\\
0&1\end{matrix}\right)d\left(\begin{matrix}e^{-\xi}& \frac{e^{-\xi}-1}{-\xi}(-w)\\ 0 &1\end{matrix}\right)-d\\
&=\left(\begin{matrix}e^\xi d e^{-\xi}-d & e^\xi d\left(\frac{e^{-\xi}-1}{\xi}(w)\right)\\ 0 & 0\end{matrix}\right)\\
\end{align*}
so that in $\Aff(V)\otimes\mathfrak{m}_A$ we find
\[
e^{(\xi,w)}\ast (0,0)=\left(e^\xi *0, e^\xi d\left(\frac{e^{-\xi}-1}{\xi}(w)\right)\right).
\]
By this and by Remark \ref{rem.injective} we find
\begin{align*}
\MC_{\Aff(V;F)\hookrightarrow\Aff(V)}(A)&=\{(\xi,w)\in \Aff^0(V)\otimes \mathfrak{m}_A\mid 
e^{(\xi,w)}\ast 0\in \Aff^1(V;F)\otimes\mathfrak{m}_A\}\\
&=\left\{(\xi,w)\in \Aff^0(V)\otimes \mathfrak{m}_A\;\middle|\; \begin{cases}
e^{\xi}\ast 0\in \End^1(V;F)\otimes\mathfrak{m}_A\\
d\left(\frac{e^{-\xi}-1}{\xi}(w)\right)\in e^{-\xi}(F\otimes A)\otimes_A\mathfrak{m}_A \end{cases} 
\!\!\!\right\}.
\end{align*}
The first condition means that $\xi\in \MC_{\End^*(V;F)\hookrightarrow\End^*(V)}(A)$, so by Lemma \ref{first.phi} the DG-vector space $F_A=e^{-\xi}(F\otimes A)$ is an element in $S_{V,F}(A)$ and the second condition then implies that
\[ \frac{e^{-\xi}-1}{\xi}(w)\in 
Z^0\left(\frac{V\otimes\mathfrak{m}_A}{F_A\otimes_A \mathfrak{m}_A}\right).
\]
\end{proof}

\begin{lemma}\label{mc-to-Z}
Let $(V,F)$ be a DG-pair. The natural transformation $\psi$ of Lemma \ref{first.psi} induces a natural
isomorphism of functors of Artin rings
\[ \psi \colon \frac{\MC_{\Aff(V;F)\hookrightarrow\Aff(V)}}{\;\exp_{\Aff^0(V;F)}\;}\xrightarrow{\;\sim\;} Z_{V,F}\;.\]
\end{lemma}

\begin{proof}
Let $A$ be a local Artin algebra, and let $(F_A,v)\in Z_{V,F}(A)$. By Lemma \ref{second.phi}, there exists an element $\xi\in \End^{0}(V)\otimes \mathfrak{m}_A$ such that $e^{-\xi}(F\otimes A)=F_A$. Since the morphism of graded vector spaces 
\[ \frac{e^{-\xi}-1}{\xi}\colon V\otimes\mathfrak{m}_A\to V\otimes\mathfrak{m}_A\]
is invertible, this shows that the morphism $\psi_A\colon \MC_{\Aff(V;F)\hookrightarrow\Aff(V)}(A)\to Z_{V,F}(A)$ is surjective. It is also invariant by the natural left action of 
$\exp_{\Aff^0(V;F)}$ on $\MC_{\Aff(V;F)\hookrightarrow\Aff(V)}(A)$. Namely, for any $(\eta,u)$ in $\Aff^0(V;F)\otimes\mathfrak{m}_A$ and any $e^{(\xi,w)}$ in $\MC_{\Aff(V;F)\hookrightarrow\Aff(V)}(A)$, let $e^{(\beta,z)}=e^{(\eta,u)}e^{(\xi,w)}$. Then we have
\begin{align*}
\left(\begin{matrix}
e^{-\beta} & \frac{e^{-\beta}-1}{\beta}(z)\\
0&1
\end{matrix}\right)&=e^{-(\beta,z)}\\
&=e^{-(\xi,w)}e^{-(\eta,u)}\\
&=\left(\begin{matrix}
e^{-\xi} & \frac{e^{-\xi}-1}{\xi}(w)\\
0&1
\end{matrix}\right)\left(\begin{matrix}
e^{-\eta} & \frac{e^{-\eta}-1}{\eta}(u)\\
0&1
\end{matrix}\right)\\
&=\left(\begin{matrix}
e^{-\xi}e^{-\eta} & e^{-\xi}\left(\frac{e^{-\eta}-1}{\eta}(u)\right)+\frac{e^{-\xi}-1}{\xi}(w)\\
0&1
\end{matrix}\right)
\end{align*}
and so
\begin{align*}
\psi_A(&e^{(\eta,u)}e^{(\xi,w)})=\left(e^{-\beta}(F\otimes A),\frac{e^{-\beta}-1}{\beta}(z)\mod e^{-\beta}(F\otimes A)\otimes_A \mathfrak{m}_A\right)\\
&=\left(e^{-\xi}e^{-\eta} (F\otimes A),e^{-\xi}\left(\frac{e^{-\eta}-1}{\eta}(u)\right)+\frac{e^{-\xi}-1}{\xi}(w)\mod e^{-\xi}e^{-\eta} (F\otimes A)\otimes_A \mathfrak{m}_A\right)\\
&=\left(e^{-\xi} (F\otimes A),e^{-\xi}\left(\frac{e^{-\eta}-1}{\eta}(u)\right)+\frac{e^{-\xi}-1}{\xi}(w)\mod e^{-\xi} (F\otimes A)\otimes_A \mathfrak{m}_A\right)\\
&=\left(e^{-\xi} (F\otimes A),\frac{e^{-\xi}-1}{\xi}(w)\mod e^{-\xi} (F\otimes A)\otimes_A \mathfrak{m}_A\right)\\
&=\psi_A(e^{(\xi,w)}),
\end{align*}
since 
$e^{-\xi}\left(\frac{e^{-\eta}-1}{\eta}(u)\right)\in e^{-\xi}(F\otimes \mathfrak{m}_A)$.
Therefore $\psi_A$ induces a surjective morphism
\[
\psi_A \colon \frac{\MC_{\Aff(V;F)\hookrightarrow\Aff(V)}(A)}{\;\exp_{\Aff^0(V;F)}(A)\;}\to Z_{V,F}(A)\;.
\]
To show that this is also injective, let $(\xi_1,w_1)$ and $(\xi_2,w_2)$ in $\MC_{\Aff(V;F)\hookrightarrow\Aff(V)}(A)$ be such that $\psi_A(e^{(\xi_1,w_1)})=\psi_A(e^{(\xi_2,w_2)})$, and let $(\eta,u)\in \Aff^0(V)\otimes \mathfrak{m}_A$ be such that $e^{(\xi_2,w_2)}=e^{(\eta,u)} e^{(\xi_1,w_1)}$. By Lemma \ref{second.phi}, $\eta$ is an element in $\End^0(V;F)\otimes \mathfrak{m}_A$ and so
\[
\psi_A(e^{(\xi_2,w_2)})=\left(e^{-\xi_1} (F\otimes A),e^{-\xi_1}\left(\frac{e^{-\eta}-1}{\eta}(u)\right)+\frac{e^{-\xi_1}-1}{\xi_1}(w_1)\mod e^{-\xi_1} (F\otimes A)\otimes_A \mathfrak{m}_A\right)
\]
Since $\psi_A(e^{(\xi_1,w_1)})=\psi_A(e^{(\xi_2,w_2)})$ we find that
\[
\frac{e^{-\eta}-1}{\eta}(u)\in (F\otimes A)\otimes_A \mathfrak{m}_A
\]
and so that $u\in (F\otimes A)\otimes_A \mathfrak{m}_A$ since $\frac{e^{-\eta}-1}{\eta}\colon F\otimes A\to F\otimes A$
is invertible. Therefore $(\eta,u)\in \Aff^0(V;F)\otimes \mathfrak{m}_A$.
\end{proof}

\begin{proposition}\label{last.prop.of.this.section}
Let $(V,F)$ be a DG-pair. The natural transformation $\psi$ of Lemma \ref{first.psi} induces 
a natural isomorphism of set-valued functors of local Artin algebras $\phi\colon \Def_{\mathfrak{Q}(V;F)}\xrightarrow{\sim} Q_{V,F}$. In particular, $Q_{V,F}$ is an infinitesimal deformation functor.
\end{proposition}

\begin{proof}
Let $(\xi,w)\in \Aff^0(V)\otimes\mathfrak{m}_A$  and 
$(\zeta,y)\in \Aff^{-1}(V)\otimes\mathfrak{m}_A$.
Since $d\zeta\colon V\otimes\mathfrak{m}_A\to V\otimes\mathfrak{m}_A$ is a morphism of complexes, 
$\frac{e^{d\zeta}-1}{d\zeta}\colon V\otimes\mathfrak{m}_A\to V\otimes\mathfrak{m}_A$
is an isomorphism of complexes, and so we have
\[
e^{(d\zeta,dy)}e^{-(\xi,w)}
=\left(\begin{matrix}
e^{d\zeta}e^{-\xi} & e^{d\zeta}\left(\frac{e^{-\xi}-1}{\xi}(w)\right)+d\left(\frac{e^{d\zeta}-1}{d\zeta}(y)\right)\\
0&1
\end{matrix}\right).
\]
The conclusion then immediately follows from Lemma \ref{lem.modellopiccolofibraomotopica}, Lemma~\ref{serve.anche.questo}, and Lemma~\ref{mc-to-Z}.
\end{proof}

\begin{remark}
The argument used in the proof of Proposition \ref{last.prop.of.this.section} shows in particular that we have a natural isomorphisms of functors
\[ \psi \colon \frac{\MC_{\Aff(V;F)\hookrightarrow\Aff(V)}}{\;\exp_{\Aff^0(V;F)}\times \exp_{dV^{-1}}\;}\xrightarrow{\sim} E_{V,F}\;.\] 
\end{remark}

\bigskip
\section{Formal periods and Abel-Jacobi maps}\label{sec:aj-maps}

We will now exhibit a Cartan calculus naturally associated with a DG-pair $(V,F)$ and will use it to define an abstract version of the period map of a K\"ahler manifold. Next, we will show that  the datum of a closed degree 0 element $v$ in $F$ defines an abstract version of the Abel-Jacobi map of a K\"ahler manifold.

\begin{definition} A \emph{formal period datum} is a Cartan calculus of the form
\begin{equation}\label{equ.cartancalculusbis} 
\xymatrix{\mathfrak{g}\ar@{^{ (}->}[d]\ar[r]&\End^*(V;F)\ar[d]\\
\cone(\mathrm{Id}_\mathfrak{g})\ar[r]&\End^*(V)} 
\end{equation}
where $(V,F)$ is a DG-pair.
\end{definition}

\begin{remark} Spelling out the definition, one sees that a formal period datum is the datum of a quadruple $(\mathfrak{g},V,F,\bi)$ where 
$\mathfrak{g}$ is a differential graded Lie algebra, $(V,F)$ is a DG-pair, and $\bi\colon \mathfrak{g}\to \End^*(V)$ is a Cartan homotopy whose boundary $\bl$ satisfies 
the condition $\bl_x(F)\subseteq F$ for every $x\in \mathfrak{g}$. 
\end{remark}
\begin{remark}\label{rem.periods}
By Proposition \ref{def.from.cartan}, the Cartan calculus  $\mathfrak{g}\xrightarrow{\,\bl\,}\End^*(V;F)\xrightarrow{}\End^*(V)$
 induces a distinguished morphism $
\mathfrak{g}\to \mathfrak{Grass}(V;F)$
in the homotopy category of DG-Lie algebras, whose associated morphism of deformation functors $\Def_{\mathfrak{g}}\to \Def_{\mathfrak{Grass}(V;F)}$ is induced by the morphism of Maurer-Cartan functors
\begin{align*}
\MC_{\mathfrak{g}}&\to \MC_{\mathfrak{Grass}(V;F)}\\
x&\mapsto (\bl_x,e^{-\bi_x}).
\end{align*}
\end{remark}

\begin{definition} By the term \emph{period map} of the formal period datum 
$(\mathfrak{g},V,F,\bi)$ we will mean both the distinguished morphism $\mathfrak{g}\to \mathfrak{Grass}(V;F)$ from Remark \ref{rem.periods}
and the natural transformation of functors 
\[ \mathcal{P}\colon \Def_{\mathfrak{g}}\to Grass_{V,F}\]
induced by it and by the 
isomorphism $\phi\colon\Def_{\mathfrak{Grass}(V;F)}\to Grass_{V,F}$  
of Proposition~\ref{prop.grassmmannfunctor}.
\end{definition}

\begin{remark}By the definition of the map $\phi\colon\Def_{g}\to Grass_{V,F}$ one immediately sees that
the period map $\mathcal{P}$ of the formal period datum 
$(\mathfrak{g},V,F,\bi)$ is explicitly given by
\begin{align*}
 \Def_{\mathfrak{g}}(A)&\to Grass_{V,F}(A)\\
 x&\mapsto e^{\bi_x}(F\otimes A),
\end{align*}
for any local Artin algebra $A$.
\end{remark}

The name period map is motivated by the following  example. 

\begin{example}[{\cite[Cor. 4.2]{FMperiods}}]\label{ex.holomorphiccartancalculus} 
Let $(A_X^{*,*},d)$ denote the Rham complex  of a  complex manifold $X$ and denote by 
$F^0_X\supseteq F^1_X\supseteq\cdots$ its Hodge filtration:
\[ F^p_X=\bigoplus_{i\ge p}A_X^{i,*}.\] 
Let $KS_X=(A_X^{0,*}(\Theta_X),-\debar,[-,-])$ be the Kodaira-Spencer differential graded Lie algebra of 
$X$; then the standard Cartan homotopy formulas  implies that the contraction operator 
\[ \bi\colon KS_X\to \End^*_{\mathbb{C}}(A_X^{*,*}),\qquad \bi_{\eta}(\omega)=\eta\contr\omega,\]
is a Cartan homotopy with boundary $\bl_\eta=[\de,\bi_{\eta}]$ the holomorphic Lie derivative.
Since $\bl_a(F^p_X)\subseteq F^p_X$ for every $a\in KS_X$ and every $p\ge 0$, we  obtain that 
$(KS_X,A_X^{*,*},F^p_X,\bi)$ is a  period datum for every $p\geq 0$,
and therefore a $p$-th period map $
\mathcal{P}_p\colon \Def_{X}\to Grass_{A_X^{*,*},F^p_X}$,
where $\Def_{X}$ denotes the functor of infinitesimal deformations of $X$ (which is canonically isomorphic to the deformation functor $\Def_{KS_X}$).
If the inclusion $F^p_X\to A_X^{*,*}$ is injective in cohomology, e.g., if the Hodge to de Rham spectral sequence of $X$ degenerates at $E_1$ as it happens for instance for compact K\"ahler manifolds, then the homotopy invariance of  $\mathfrak{Grass}$ from Remark \ref{rem.homotopy-invariance-grass} implies that there exists a natural 
isomorphism of Grassmann functors $Grass_{A_X^{*,*},F^p_X}\cong Grass_{H^*(A_X^{*,*}),H^*(F^p_X)}$. In this case the map $\mathcal{P}_p$ is therefore a morphism
\[ \mathcal{P}_p\colon \Def_X\to Grass_{H^*(A_X^{*,*}),H^*(F^p_X)},\]
and it is proved in \cite{FMperiods} that this map is precisely the Griffiths's $p$-th period map for infinitesimal deformations of $X$.
\end{example}

Under certain assumption, 
formal period maps can be conveniently used in order to prove the vanishing to obstructions to deformations.

\begin{proposition}\label{prop.abstractkodairaprinciple} 
Let $(\mathfrak{g},V,F,\bi)$ be a period datum such that the map $H^*(F)\to H^*(V)$ is injective. Then, the obstructions of the functor $\Def_\mathfrak{g}$ are contained in the kernel of 
\[H^2(\bi)\colon H^2(\mathfrak{g})\to \Hom^{1}_{\K}(H^*(F),H^*(V/F)).\]
Moreover, if $\bi\colon H^*(\mathfrak{g})\to \Hom^{*}_{\K}(H^*(F),H^*(V/F))[-1]$ is injective, then 
$\mathfrak{g}$ is homotopy abelian.
\end{proposition}

\begin{proof} Since the inclusion $F\hookrightarrow V$ is injective in cohomology, according to Proposition~\ref{prop.injectiveincohomologygrassabelian},  
the differential graded Lie algebra $\mathfrak{Grass}(V,F)$ is homotopy abelian. By Remark \ref{rem.cartanincohomologia}, the morphism
\[
H^*(\mathfrak{p})\colon H^*(\mathfrak{g})\to H^*(\mathfrak{Grass}(V,F))
\]
induced by the period map $\mathcal{P}\colon \mathfrak{g}\to \mathfrak{Grass}(V,F)$ is identified with the morphism 
\[ H^*(\bi)\colon H^*(\mathfrak{g})\to H^*(\Coker(\End^*(V;F)\hookrightarrow \End^*(V))[-1])\cong \Hom^{*-1}_{\K}(H^*(F),H^*(V/F))\;.
\]
In particular, the map $\bi\colon H^2(\mathfrak{g})\to \Hom^{1}_{\K}(H^*(F),H^*(V/F))$ maps obstructions of $\Def_{\mathfrak{g}}$ to obstructions of $Grass_{V,F}$, which are trivial since is $\mathfrak{Grass}(V,F)$ is homotopy abelian.
Finally, if the map $H^*(\bi)\colon H^*(\mathfrak{g})\to \Hom^{*-1}_{\K}(H^*(F),H^*(V/F))$ is injective, the homotopy abelianity of $\mathfrak{g}$ follows from Proposition~\ref{prop.iniettivoincohomologiaversohomotopyabelian}.
\end{proof}

\begin{example}[{\cite{algebraicBTT}, Thm. B}] Let $X$ be a smooth projective complex manifold. Then the obstruction to infinitesimal deformations of $X$ are contained in the intersection
\[
\bigcap_{p}\ker\left\{H^2(\bi)\colon H^2(X,\Theta_X)\to\Hom^{1}_{\K}(H^*(F^p_X,H^*(A_X^{*,*}/F^p_X))\right\}.
\]
In particular, if $X$ is a Calabi-Yau manifold then $\Def_X$ is unobstructed. Actually more is true: if $X$ is a Calabi-Yau manifold then $KS_X$ is homotopy abelian.
\end{example}

\begin{definition} A \emph{formal Abel-Jacobi datum} is a tuple $(\mathfrak{g},\tilde{\mathfrak{g}},V,F,v,\bi)$ where
$(\mathfrak{g},V,F,\bi)$ is a formal period datum, $v\in Z^0(F)$ is a $0$-cocycle in $F$, and
$\tilde{\mathfrak{g}}$ is a DG-Lie subalgebra of $\mathfrak{g}$ such that $\bi_x(v)=0$ for every $x\in \tilde{\mathfrak{g}}$.
\end{definition}

\begin{lemma}\label{lemma.lv}
Let $\bi:\mathfrak{g}\to \mathrm{End}(V)[-1]$ be a Cartan homotopy and let $v$ be a degree zero closed element in $V$. Then 
\begin{align*}
{\bi}^v:\mathfrak{g}&\to\mathrm{Aff}(V)[-1]\\
x&\mapsto (\bi_x,-\bi_x(v))
\end{align*}
is a Cartan homotopy. The corresponding Lie derivative is
\begin{align*}
{\bl}^v:\mathfrak{g}&\to\mathrm{Aff}(V)\\
x&\mapsto (\bl_x,-\bl_x(v))
\end{align*}
\end{lemma}
\begin{proof}
The linear map $\bi^v$ is the composition of the Cartan homotopy $\bi$ with the DG-Lie morphism $\sigma_v$ of 
Remark~\ref{rem.sezionidistorte}, hence it is a Cartan homotopy. The corresponding coboundary  is the composition of $\bl$ with $\sigma_v$.
\end{proof}

\begin{proposition}\label{proposition.cartan-from-v}
Let $(\mathfrak{g},\tilde{\mathfrak{g}},V,F,v,\bi)$ be a formal Abel-Jacobi datum.  Then,
the diagram
\begin{equation}\label{equ.abeljacobicube}
\xymatrix@!0{
& & \mathfrak{g} \ar[rrrr]^-{{\bl}^v}\ar'[d][dd]
& & & &{\mathrm{Aff}}(V;F) \ar[dd]
\\
\tilde{\mathfrak{g}} \ar[urr]\ar[rrrr]^-{
{\bl}\big\vert_{\tilde{\mathfrak{g}}}}\ar[dd]
& & & &\mathrm{End}(V;F) \ar[urr]_{\sigma_0}\ar[dd]
\\
& &\text{\phantom{i}\small{\rm cone}}(\mathrm{Id}_{\mathfrak{g}}) \ar'[rr][rrrr]
& &  & &{\mathrm{Aff}}(V)
\\
\text{\small{\rm cone}}(\mathrm{Id}_{\tilde{\mathfrak{g}}})\ar@{-}[rrrr]\ar[urr]
& & & &\mathrm{End}(V) \ar[urr]_{\sigma_0}
}
\end{equation}
where $\tilde{\mathfrak{g}}\to \mathfrak{g}$ is the inclusion, is a morphism of Cartan calculi.
\end{proposition}
\begin{proof}
One only needs to prove the commutativity of the upper and of the lower face of the cube. 
The commutativity of the lower face of the cube is equivalent to the commutativity of the the diagram of graded vector spaces
\[
\xymatrix{
\tilde{\mathfrak{g}}\ar[d]\ar[r]^-{{\bi}\big\vert_{\tilde{\mathfrak{g}}}}&\mathrm{End}(V)[-1]\ar[d]^{\sigma_0[-1]}\\
\mathfrak{g}\ar[r]^-{{\bi}^v}&\mathrm{Aff}(V)[-1]\,.
}
\]
For $x$ in $\tilde{\mathfrak{g}}$ we have $\bi_x(v)=0$ and then $\bi^v(x)=
(\bi_x,-\bi_x(v))= (\bi_x,0)=\sigma_0(\bi_x)$. The commutativity of the upper face follows similarly, by noticing that \[
\bl_x(v)=d_V(\bi_x(v))-(-1)^{\overline{x}}\bi_x(d_Vv)+\bi_{d_{\mathfrak{g}}x}(v)=0,
\]
since $v$ is closed and both $x$ and $d_{\mathfrak{g}}x$ are in $\tilde{\mathfrak{g}}$.
\end{proof}

\begin{proposition}\label{prop.tangent-obstruction-aj}
A formal Abel-Jacobi datum $(\mathfrak{g},\tilde{\mathfrak{g}},V,F,v,\bi)$ induces  a natural morphism 
\[
\mathcal{AJ}\colon hofib(\tilde{\mathfrak{g}}\hookrightarrow \mathfrak{g})\to \mathfrak{J}(V,F)\]
in the homotopy category of DG-Lie algebras. The linear morphism induced in cohomology by the map $\mathcal{AJ}$ is canonically identified with the morphism
\begin{align*}
H^{*-1}(\mathfrak{g}/\tilde{\mathfrak{g}})&\to H^{*-2}(V/F)\\
[x]&\mapsto -[\bi_{\tilde{x}}(v)\mod F].
\end{align*}
where $\tilde{x}$ is any representative of $[x]$ in $\mathfrak{g}$. 
\end{proposition}

\begin{proof}
The first part of the statement is immediate from Remark \ref{rem.morphism-of-Cartan-calculi-to-span} and from the definition of $\mathfrak{J}(V,F)$, i.e., from Definition \ref{def.jacobian-dgla}. The second part is a particular case of Proposition \ref{prop.double-fiber-in-cohomology}, together with the quasi-isomorphisms 
\begin{align*}
\Aff(V)/(\Aff(V;F)+\End(V))&\xrightarrow{\sim}V/F\\
[(\xi,v)]&\mapsto [v \mod F]
\end{align*}
from Lemma \ref{lemma.J-easy}.
\end{proof}

\begin{remark}
The above construction tells us in particular that the element $\bi_{\tilde{x}}v$ is closed in the quotient $V/F$ and that it is independent of the choice of the representative $\tilde{x}$. It is a simple exercise to verify directly both of these properties. Namely, to say that $\bi_{\tilde{x}}v$ is closed in $V/F$ amounts to saying that $d_{V}(\bi_{\tilde{x}}v)$ is an element of $F$, and we have
\begin{align*}
d_{V}(\bi_{\tilde{x}}v)&=-\bi_{d_{L}\tilde{x}}(v)+\bl_{\tilde{x}}(v)-(-1)^x \bi_{\tilde{x}}(d_{V}v)\\
&=-\bi_{d_{L}\tilde{x}}(v)+\bl_{\tilde{x}}(v).
\end{align*}
since by hypothesis $v$ is a closed element in $V$. Also, the element $\bl_{\tilde{x}}(v)$ lies in $F$ since $\bl_{\tilde{x}}^v$ is an element of $\mathrm{Aff}(V;F)$. So we are reduced to show that also $\bi_{d_{\mathfrak{g}}\tilde{x}}(v)$ lies in $F$. Since $\tilde{x}$ is a representative in $\mathfrak{g}$ of the class $[x]$ in $\mathfrak{g}/\tilde{\mathfrak{g}}$, the element $\tilde{x}$ is in particular closed in the quotient $\mathfrak{g}/\tilde{\mathfrak{g}}$, i.e. $d_{\mathfrak{g}}\tilde{x}$ is an element of $\tilde{\mathfrak{g}}$. Therefore, $\bi_{d_{\mathfrak{g}}\tilde{x}}(v)=0$ by definition of $\tilde{\mathfrak{g}}$.
To see that the cohomology class $[\bi_{\tilde{x}}(v)\mod F]$ is indeed well defined, let $\tilde{x}$ be a representative of the zero class in $H^{n-1}(\mathfrak{g}/\tilde{\mathfrak{g}})$. This means that $\tilde{x}=d_{\mathfrak{g}}y+z$ for some $y\in \mathfrak{g}$ and $z\in \tilde{\mathfrak{g}}$. Then $\bi_{\tilde{x}}(v)=\bi_{d_{\mathfrak{g}}y}v=\bl_y(v)-d_{V}(\bi_y(v))$, since $z\in \tilde{\mathfrak{g}}$ and $v$ is closed. Now $\bl_y(v)\in F$ since $\bl_y^v\in\mathrm{Aff}(V)$ and so $\bi_{\tilde{x}}(v)$ is exact in $V/F$.
\end{remark}

\begin{definition}
We will call \emph{formal Abel-Jacobi map} associated with the formal Abel-Jacobi datum $(\mathfrak{g},\tilde{\mathfrak{g}},V,F,v,\bi)$ both the map $\mathcal{AJ}\colon hofib(\tilde{\mathfrak{g}}\hookrightarrow \mathfrak{g})\to \mathfrak{J}(V,F)$ in the homotopy category of DG-Lie algebras and the associated morphism of deformation functors
\[
\mathcal{AJ}\colon\Def_{\tilde{\mathfrak{g}}\hookrightarrow \mathfrak{g}}\to Jac(V;F).
\]
\end{definition}

\bigskip
\section{The infinitesimal Abel-Jacobi map of a submanifold}

In the previous section we defined an abstract notion of Abel-Jacobi map as a morphism of deformation theories, described by a morphism in the homotopy category of differential graded Lie algebras. 
Here we describe a geometric situation where the above construction applies, and which motivates the name ``formal Abel-Jacobi map'' we gave to this construction.
In this section we work over the field $\C$ of complex numbers.

For a smooth complex manifold $X$ of dimension $n$, we shall denote by   
$D^{p,q}_X$ the space of currents of type $(p,q)$ on $X$. 
We recall, see \cite[Section 3.1]{G-H}, that 
$D_X^{*,*}=\oplus_{p,q}D^{p,q}_X$ is the  topological dual, in the $C^{\infty}$ topology, of the space $A_{X,c}^{*,*}$ of differential forms  with compact support on the smooth complex manifold $X$, and   
$D^{p,q}_X\subseteq D_X^{*,*}$ is the subspace  of functionals vanishing on differential forms of type $(i,j)$, for every 
$(i,j)\not=(n-p,n-q)$.

Recall also that there exists a natural inclusion 
$A^{p,q}_X\subseteq D^{p,q}_X$, where
a differential form $\omega\in A^{p,q}_X$ is identified with the current (which we denote by the same symbol $\omega$) 
\[ \omega\colon A_{X,c}^{*,*}\to \C,\qquad \psi\mapsto \int_X \omega\wedge \psi\;.\]

\begin{lemma}\label{lem.extendingderivations} 
Let $\alpha$ be a continuous derivation of bidegree $(p,q)$ of the de Rham algebra $A^{*,*}_X$ such that 
$\int_X\circ \alpha=0$. Then $\alpha$ can be naturally extended to a linear operator $\hat{\alpha}$ of bidegree $(p,q)$
on the space of currents by the rule  
\[
\hat{\alpha}(T)\colon \eta\mapsto T(\alpha^*(\eta)), 
\]
where $\alpha^*\colon A^{i,j}_{X,c}\to A^{i+p,j+q}_{X,c}$ is defined by $\alpha^*(\eta)=-(-1)^{(p+q)(i+j+p+q)}\alpha(\eta)$.
\end{lemma}
\begin{proof} By a  standard use of bump functions we have 
that every derivation of $A^{*,*}_X$ is a local operator and therefore preserves the subspace of differential forms with compact support.

One clearly sees that $\hat{\alpha}$ is a linear operator from $D^{i,j}_{X}$ to $D^{i+p,j+q}_{X}$ so the only nontrivial part to be shown is that $\hat{\alpha}$ actually extends $\alpha$, i.e., that for every $\omega\in A^{*,*}_X$ and $\eta\in A^{*,*}_{X,c}$ we have
\[
\hat{\alpha}(\omega)\eta=\int_X\omega\wedge\alpha^*(\eta)=\int_X\alpha(\omega)\wedge\eta.
\]
This immediately follows from the Leibniz rule $\alpha(\omega\wedge\eta)=\alpha(\omega)\wedge\eta-\omega\wedge\alpha^*(\eta)$ together with the assumption $\int_X\circ \alpha=0$.
\end{proof}

\begin{lemma}\label{lemma.anticommutators}
Let $\alpha,\beta$ be continuous derivations of the de Rham algebra $A^{*,*}_X$ such that 
$\int_X\circ \alpha=\int_X\circ \beta=0$. Then, in the notation of Lemma~\ref{lem.extendingderivations} we have 
\[ \widehat{[\alpha,\beta]}=[\hat{\alpha},\hat{\beta}]\;.\] 
\end{lemma}

\begin{proof} One is reduced to considering homogeneous components and 
the conclusion is a straightforward consequence of the easy formulas
\[ (\hat{\alpha}\hat{\beta}T)(\eta)=(\hat{\beta}T)(\alpha^*\eta)=T(\beta^*\alpha^*\eta),\qquad 
\alpha^*\beta^*=-(-1)^{\bar{\alpha}\;\bar{\beta}}(\alpha\beta)^*\;.\]
\end{proof}

The above construction applies in particular for the derivations $\de,\debar$ and $\bi_{\xi}$, with 
$\xi \in A^{0,*}_X(\Theta_X)$;  
we get the operators
\[ \widehat{\de}\colon D^{p,q}_X\to D^{p+1,q}_X,\quad \widehat{\debar}\colon D^{p,q}_X\to D^{p,q+1}_X,\quad
\widehat{\bi_{\eta}}\colon D^{p,q}_X\to D^{p-1,q+\bar{\xi}}_X\;.\]
The same applies also to the holomorphic Lie derivative $\bl_\xi=[\partial,\bi_\xi]$, giving operators
\[ \widehat{\bl_{\xi}}\colon D^{p,q}_X\to D^{p,q+\bar{\xi}}_X\;.\]

\begin{corollary}[holomorphic Cartan homotopy formulas] 
Let $\xi,\zeta$ be elements in $A^{0,*}_X(\Theta_X)$. Then we have the following identities between linear operators on $D^{*,*}_X$:
\[ \left[\widehat{\partial},\widehat{\bi_\xi}\right]=\widehat{\bl_\xi},\qquad
\left[\widehat{\bi_\xi},\widehat{\bl_\zeta}\right]=\widehat{\bi_{[\xi,\zeta]}},
\qquad  
\left[\widehat{\bi_\xi},\widehat{\bi_\zeta}\right]=0\,,\qquad 
\left[\widehat{\overline{\de}},\widehat{\bi_\xi}\right]=\widehat{\bi_{\overline{\de}\xi}}\,.\]
\end{corollary}
\begin{proof}
The above identities are immediately obtained by Lemma~\ref{lemma.anticommutators} and by the usual 
holomorphic Cartan homotopy formulas, see, e.g., \cite[Lemma 4.1]{FMperiods}.
\end{proof}

From now on, for simplicity of notation, we omit the 
hat symbol in all the operators involved in Cartan's formulas and we write 
\[ \de\colon D^{p,q}_X\to D^{p+1,q}_X,\quad \debar\colon D^{p,q}_X\to D^{p,q+1}_X,\quad
\bi_{\xi}\colon D^{p,q}_X\to D^{p-1,q+\bar{\xi}}_X,\quad 
\bl_{\xi}\colon D^{p,q}_X\to D^{p,q+\bar{\xi}}_X\;.\]

The above considerations immediately give the following result. 
\begin{corollary}\label{cor.cartancalculusforcurrents}
Let $X$ be a smooth complex manifold, and let $KS_X=(A_X^{0,*}(\Theta_X),-\debar,[-,-])$ be its Kodaira-Spencer DG-Lie algebra. 
Then, for every integer $p$, the contraction map 
\[
\bi\colon KS_X\to \End^*(D_X^{*,*}[2p])
\]
is a Cartan homotopy with boundary $\bl$ equals to  the holomorphic Lie derivative on currents.
\end{corollary}

\begin{proof} The case $p=0$ is a direct consequence of Cartan homotopy formulas, with the same computations of 
Example~\ref{ex.holomorphiccartancalculus}. The tautological identification 
$\End^*(D_X^{*,*})\simeq \End^*(D_X^{*,*}[2p])$ is an isomorphism of differential graded Lie algebras and then, 
also the contraction map $\bi\colon KS_X\to \End^*(D_X^{*,*}[2p])$ is a Cartan homotopy.
\end{proof}

For every $p$ we have an inclusion 
of complexes 
\[ (A_X^{p,*},\debar)\to (D^{p,*}_X,\debar)\]
which, by $\debar$-Poincar\'e lemma for currents (see e.g. \cite[p. 383]{G-H}) is a quasi-isomorphism and therefore for every $p\ge 0$ the   
inclusion of  complexes 
\[(\oplus_{i\ge p} A_X^{i,*}, \de+\debar)\to  (\oplus_{i\ge p} D_X^{i,*}, \de+\debar)\]
is a quasi-isomorphism.  

\begin{lemma} In the notation above, if $X$ is compact K\"{a}hler, then for every $p\ge 0$, the 
inclusion of complexes 
\[(\oplus_{i\ge p} D_X^{i,*}, \de+\debar)\to  (\oplus_{i\ge 0} D_X^{i,*}, \de+\debar)\]
is injective in cohomology.  
\end{lemma} 

\begin{proof} Immediate from the analogous statement for differential forms.\end{proof}

Every closed complex submanifold $Z\subseteq X$ of  pure codimension $p$ determines a current 
$\int_{Z}\in D^{p,p}_X$:
\[ \int_{Z}\colon A_{X,c}^{*,*}\to \C,\qquad
\psi\mapsto \int_{Z}\psi\;.
\] 
Moreover, since the integral on $Z$ of every exact form in $A^{*,*}_{X,c}$ vanishes, the current $\int_Z$
is $\partial$-and-$\overline{\partial}$-closed: in other words 
$\int_Z$ is an element in $Z^{2p}(D^{*,*}_X)=Z^0(D^{*,*}_X[2p])$.

\begin{lemma}\label{for-is-staisfied-i}
Let $Z\subseteq X$ be a closed  complex submanifold of pure codimension $p$ and let $KS_{(X,Z)}$ be the DG-Lie subalgebra of $KS_X$ defined by
\[
KS_{(X,Z)}:=
A^{0,*}_X(\Theta_X(-\log Z))=\ker(A_X^{0,*}(\Theta_X)\to A^{0,*}_Z(N_{Z/X})),
\]
where $N_{Z/X}$ is the normal sheaf to $Z$ in $X$. Then we have
\[
\bi_\xi \int_Z=0\qquad 
\]
for any $\xi$ in $KS_{(X,Z)}$.
\end{lemma}

\begin{proof} 
Let $\xi$ be an element in $A^{0,q}_X(\Theta_X(-\log Z))$. Then, for any $\eta\in A^{n-p+1,n-p-q}_{X,c}$, we have
\[
\left(\bi_\xi\int_Z\right)\eta=-\int_Z \bi_\xi \eta
\]
and then it is sufficient to prove that $(\bi_\xi \eta)_{|Z}=0$.
Around a point $x\in Z$ we have a holomorphic chart $U$ and  holomorphic coordinates $\{z^1,\dots,z^n\}$  such that $Z\cap U$ is  defined by the equations $z^1=0,\dots, z^p=0$. 
By linearity it is sufficient to consider the case  where $\eta$ has the form
\[
\eta=dz^{i_1}\wedge\cdots \wedge dz^{i_{n-p+1}}\wedge \mu,\qquad i_1<\cdots<i_{n-p+1},\; \mu\in A^{0,n-p-q}_{U,c}.\]
The element $\xi$ may be written  as 
\[ \xi=\sum_{j=1}^n\frac{\de~}{\de z^j}\otimes \xi_j,\qquad \xi_j\in A^{0,q}_{U},\quad\text{and}\quad (\xi_j)_{|Z}=0 \quad \text{for}\;j\le p\; \]
and therefore
\[ \bi_{\xi}\eta=(-1)^{q(n-p+1)}\sum_{j=1}^n\left(\frac{\de~}{\de z^j}\contr dz^{i_1}\wedge\cdots \wedge dz^{i_{n-p+1}}\right)\wedge 
\xi_j\wedge\mu\;.\]
Since $dz^1_{|Z}=0,\ldots, dz^p_{|Z}=0$ we have $(\bi_\xi\eta)_{|Z}=0$ unless $i_1\le p$ and
$i_2=p+1,\ldots,i_{n-p+1}=n$, while in this particular case  we also have 
\[ (\bi_{\xi}\eta)_{|Z}=(-1)^{q(n-p+1)}
dz^{p+1}\wedge\cdots \wedge dz^{n}\wedge 
\xi_{i_1}\wedge\mu_{|Z}=0\;.\]
\end{proof}

As we already remarked, it is well known that the DG-Lie algebra $KS_X$ controls the infinitesimal deformations of $X$. Similarly, since $Z$ is smooth, the DG-Lie algebra $KS_{(X,Z)}$ controls the infinitesimal deformations of the pair $(X,Z)$, see e.g. \cite{IaconoDP} and references therein. It follows that the (homotopy class of the) DG-Lie algebra
\[
\mathfrak{Hilb}_{X/Z}:=hofib(KS_{(X,Z)}\hookrightarrow KS_X)
\]
controls the infinitesimal  embedded deformations of $Z$ inside $X$, see \cite{donarendiconti,ManettiSemireg} for a complete proof.  In other words, one has an equivalence of infinitesimal deformation functors
\[
\Hilb_{X/Z}\cong \Def_{\mathfrak{Hilb}_{X/Z}},
\]
where $\Hilb_{X/Z}$ is the infinitesimal neighborhood of the Hilbert scheme of $X$ at the submanifold $Z$.  More in detail, if one denotes by 
$i^*\colon \mathcal{A}^{0,0}_X\to \mathcal{A}^{0,0}_Z$ the restriction map between the sheaves of differentiable functions, then  
according to \cite[Thm. 5.2]{ManettiSemireg} and using the description of the functor 
$\Def_{\mathfrak{Hilb}_{X/Z}}$ from Lemma~\ref{lem.modellopiccolofibraomotopica},  to any local Artin $\C$-algebra $B$ and every Maurer-Cartan element  $\xi$ in 
\[ \MC_{\mathfrak{Hilb}_{X/Z}}(B)=\{ \xi\in A_X^{0,0}(\Theta_X)\otimes\mathfrak{m}_B\mid e^\xi\ast 0\in 
A_X^{0,1}(\Theta_X(-\log Z))\otimes\mathfrak{m}_B\}\]
it is associated the subscheme
$Z_B\subseteq X\times Spec(B)$ with ideal sheaf $\mathcal{O}_X\otimes B\cap e^{-\xi}(\ker i^*\otimes B)$.
Geometrically, $Z_B$ is the image of $Z\times Spec(B)$ under the formal diffeomorphism of $X\times Spec(B)$ obtained
by integration of the vector field $-\xi$.

\begin{definition}
Let $(X,Z)$ be a pair consisting of a complex manifold  $X$ and of a closed complex submanifold $Z\subset X$ of pure codimension $p$. The formal Abel-Jacobi datum of the pair $(X,Z)$ is defined as
\[ \left(KS_X,\,KS_{(X,Z)},\, D^{*,*}_X[2p],\, D^{\ge p,*}_X[2p],\,\int_Z\,,\,\bi\right)\,. \]
We will denote by $\mathcal{AJ}_{(X,Z)}$ the associated Abel-Jacobi maps, both at the level of the homotopy category of DG-Lie algebras,
\[
\mathcal{AJ}_{(X,Z)}\colon \mathfrak{Hilb}_{X/Z}\to \mathfrak{J}(D^{*,*}_X[2p], D^{\ge p,*}_X[2p])
\]
and at the level of infinitesimal deformation functors
\[
\mathcal{AJ}_{(X,Z)}\colon \Hilb_{X/Z}\to Jac^{2p}_X.
\]
where we have written $Jac^{2p}_X$ for $Jac(D^{*,*}_X[2p], D^{\ge p,*}_X[2p])$.
\end{definition}

By Lemma \ref{lemma.J-easy}, we know that the Jacobian DG-Lie algebra $\mathfrak{J}(D^{*,*}_X[2p], D^{\ge p,*}_X[2p])$ which we have seen is homotopy abelian and quasi-isomorphic to the quotient complex 
\[\frac{D^{*,*}_X[2p-2]}{D^{\ge p,*}_X[2p-2]}=
D^{<p,*}_X[2p-2]\,.\]
Therefore the Abel-Jacobi map $\mathcal{AJ}_{(X,J)}$ can equivalently be seen as a morphism
\[
\mathcal{AJ}_{(X,Z)}\colon \mathfrak{Hilb}_{X/Z}\to D^{<p,*}_X[2p-2]
\]
in the homotopy category of DG-Lie algebras. Since $D^{p,*}_X$ is a fine resolution of the sheaf $\Omega_X^p$, the functor 
$Jac^{2p}_X$ is pro-represented by the germ at $0$ of the hypercohomology group
\[ H^1(D^{<p,*}_X[2p-2])=\mathbb{H}^{2p-1}(X;\mathcal{O}_X\to\Omega^1_X\to\cdots \to\Omega^{p-1}_X)\,.\]

\begin{proposition}
The differential of the Abel-Jacobi map \[
\mathcal{AJ}_{(X,Z)}\colon \Hilb_{X/Z}\to Jac^{2p}_X
\]
is the linear map  
\begin{align*}
AJ_1\colon H^{0}(Z;N_{X/Z})&\to H^{2p-1}(D^{<p,*}_X)\cong \mathbb{H}^{2p-1}(X;\mathcal{O}_X\to\Omega^1_X\to\cdots \Omega^{p-1}_X)\\
[x]&\mapsto -\left[\bi_{\tilde{x}}\int_Z\mod D^{\ge p,*}_X[2p]\right],
\end{align*}
while  the obstruction map is
\begin{align*}
AJ_2\colon H^{1}(Z;N_{X/Z})&\to H^{2p}(D^{<p,*}_X)\cong \mathbb{H}^{2p}(X;\mathcal{O}_X\to\Omega^1_X\to\cdots \Omega^{p-1}_X)\\
[x]&\mapsto -\left[\bi_{\tilde{x}}\int_Z\mod D^{\ge p,*}_X[2p]\right],
\end{align*}
where $\tilde{x}$ is any representative of $[x]$ in $A^{0,*}_X(\Theta_X)$. 
In particular, since the functor $Jac^{2p}_X$ is unobstructed,  every obstruction 
to infinitesimal deformations of $Z$ in $X$ is contained in the kernel of $AJ_2$. 
\end{proposition}

\begin{proof}
The differential and the obstruction map for $\mathcal{AJ}_{(X,Z)}$ are given by
\[
H^i(\mathcal{AJ}_{(X,Z)})\colon H^i(\mathfrak{Hilb}_{X/Z}) \to H^i(\mathfrak{J}(D^{*,*}_X[2p], D^{\ge p,*}_X[2p]))
\]
for $i=1,2$, respectively. One concludes by Proposition~\ref{prop.tangent-obstruction-aj}.
\end{proof}

By the explicit description given above,  the image of the tangent map  is contained in the image of the natural map 
\begin{equation}\label{equ.inclusion1}
H^p(\Omega^{p-1}_X)=\mathbb{H}^{2p-1}(X;0\to \cdots 0\to\Omega^{p-1}_X)\to 
\mathbb{H}^{2p-1}(X;\mathcal{O}_X\to\Omega^1_X\to\cdots \Omega^{p-1}_X)\end{equation}
and the image of the obstruction map is contained in the image of  
\begin{equation}\label{equ.inclusion2}
H^{p+1}(\Omega^{p-1}_X)=\mathbb{H}^{2p}(X;0\to \cdots 0\to\Omega^{p-1}_X)\to 
\mathbb{H}^{2p}(X;\mathcal{O}_X\to\Omega^1_X\to\cdots \Omega^{p-1}_X)\,.
\end{equation}

When the Hodge to de Rham spectral sequence degenerate at $E_1$, e.g. if $X$ is compact K\"{a}hler, then 
\eqref{equ.inclusion1} and   \eqref{equ.inclusion2} are injective maps and then we recover the differential $AJ_1$ as the usual infinitesimal Abel-Jacobi map, see either \cite[pag. 28]{Green}\footnote{Notice the misprint in the last line.} or \cite[Lemme 12.6]{Voisin} 
\[ H^{0}(Z;N_{X/Z})\to H^p(X,\Omega_X^{p-1}),\]
and the obstruction map as the usual semiregularity map \cite{bloch,semireg2011,ManettiSemireg}.
\[ H^{1}(Z;N_{X/Z})\to H^{p+1}(X,\Omega_X^{p-1})\,.\]

\appendix

\bigskip
\section{Direct $L_{\infty}$ maps}
\label{app.iterati}
Using an explicit $L_\infty$ quasi-isomorphism 
$TW^{[2]}(\mathcal{J}(V,F)_{\bullet\bullet}) \to V/F[-2]$ defined by means of iterated integrals, one can exhibit an explict morphism of Maurer-Cartan functors
\[
\mathcal{AJ}_{\MC}\colon \MC_{\tilde{\mathfrak{g}}\hookrightarrow \mathfrak{g}}\to \MC_{V/F[-2]}
\]
lifting the formal Abel-Jacobi map $\mathcal{AJ}\colon\Def_{\tilde{\mathfrak{g}}\hookrightarrow \mathfrak{g}}\to Jac(V;F)$.

\medskip

For a field $\K$ of characteristic $0$, we define the  linear maps 
$\int_n\colon \K[t]^{\otimes n}\to \K$, $n>0$,  by the formulas
\[\int_np_1\otimes\cdots\otimes p_n=\int_{0\le t_1\le \cdots\le t_n\le 1}
p_1(t_1)\cdots p_n(t_n)dt_1\cdots dt_n\;.\]
Equivalently, $\int_n$ is the unique $\K$-linear map such that 
\[ \int_nt^{a_1}\otimes\cdots\otimes t^{a_n}=\frac{1}{\;\prod_{h=1}^n(1+\sum_{i=1}^h a_i)\;}\,,
\qquad a_1,\ldots,a_n\ge 0\,.\]

\begin{lemma}
For  $n\ge 1$ and  $p,q_1,\ldots,q_n\in \K[t]$ such that $p(0)=0$ we have: 
\begin{enumerate}

\item 
\[\int_{n+1}p'\otimes q_1\otimes\cdots\otimes q_n=
\int_{n}pq_1\otimes q_2\otimes \cdots\otimes q_n\;,\]
where $p'$ is the usual derivative of $p$ in $\K[t]$;

\item if $0<i<n$,  then 
\[\int_{n+1}\cdots\otimes q_i\otimes p'\otimes q_{i+1}\otimes \cdots=
-\int_{n}\cdots q_ip\otimes q_{i+1}\cdots\; +\int_{n}\cdots q_i\otimes pq_{i+1}\cdots\;;\]

\item if $p(1)=0$, then 
\[\int_{n+1}q_1\otimes\cdots\otimes q_n\otimes p'=-\int_{n}q_1\otimes\cdots\otimes q_np\;.\]

\end{enumerate}
\end{lemma}

\begin{proof} The proof is completely straightforward. Probably the simplest way is by introducing 
the multilinear maps $\Phi_n\colon \K[t]\times\cdots\times \K[t]\to \K[t]$:
\[\Phi_{n}(p_1,\ldots,p_n)(t)=\int_{0\le t_1\le \cdots\le t_n\le t}
p_1(t_1)\cdots p_n(t_n)dt_1\cdots dt_n\;,\]
and using integration by parts together with the obvious identities:
\[ \int_{n}p_1\otimes\cdots\otimes p_n=\Phi_{n}(p_1,\ldots,p_n)(1),\qquad 
\Phi_1(p)(t)=\int_0^tp(s)ds,\]
\[\Phi_{n}(p_1,\ldots,p_n)(t)=\int_0^t\Phi_{n-1}(p_1,\ldots,p_{n-1})(s)p_n(s)ds\;,\]
\[\Phi_{n}(p_1,\ldots,p_n)=\Phi_{i}(\Phi_{n-i}(p_{1},\ldots,p_{n-i})p_{n-i+1},\ldots, p_n)\;.\]
\end{proof}

Given a differential graded associative algebra $A$ and a left differential graded $A$-module $M$, we functorially associate a new differential graded associative algebra $B(A,M)$ in the following way: first we consider  
the direct sum of complexes $A\oplus M$ as a differential graded associative algebra equipped with the product  
\[ (a,m)(b,n)=(ab,an),\qquad a,b\in A,\; n,m\in M\;.\]
Notice that $MA=0$, the natural inclusion $A\to A\oplus M$ is a morphism of differential graded associative algebras, and its 
Thom-Whitney homotopy fiber is naturally isomorphic to the algebra
\[ B(A,M)=\{(a(t),m(t))\in A[t,dt]\oplus M[t,dt]\mid a(0)=0,\; m(0)=m(1)=0\}\,.\]

As a graded vector space we have
\[ B(A,M)=R\oplus S=R_0\oplus R_1\oplus S_0\oplus S_1,\]
where $R=R_0\oplus R_1$, $S=S_0\oplus S_1$, 
\[ R_0=\{p\in A[t]\mid p(0)=0\},\qquad R_1=A[t]dt\]
\[ S_0=\{p\in M[t]\mid p(0)=p(1)=0\},\qquad S_1=M[t]dt\;.\]

In the next theorem we shall denote by $\overline{T}(V)$ the reduced tensor coalgebra generated by 
the graded vector space $V$; the bar construction gives a natural DG-coalgebra structure on 
$\overline{T}(B(A,M)[1])$.

\begin{theorem}\label{thm.iterated} 
Consider $M[-1]$ as a differential associative algebra with trivial product, then 
the linear map $\int_{\infty}\colon \overline{T}(B(A,M)[1])\to M$ of degree 0, defined 
by setting
\[ \int_{\infty}\colon R_1^{\otimes n}\otimes S_1\to M,\quad n\ge 0,\]
\[ \int_{\infty}a_1q_1(t)dt\otimes\cdots\otimes a_nq_n(t)dt\otimes  mp(t)dt
=a_1\cdots a_n m \int_{n+1}q_1\otimes\cdots\otimes q_n\otimes p\;,\]
and extending by $0$ to the other direct summands of $\overline{T}(B(A,M)[1])$, gives an $A_{\infty}$ quasi-isomorphism 
$\int_{\infty}\colon B(A,M)\to M[-1]$.
\end{theorem}

\begin{proof} Let's denote by the same letter $d$ the differentials of $A$, $R$, $M$ and $S$. 
We shall denote by $Q=Q_1+Q_2$ the differential on the reduced
tensor coalgebra 
\[\overline{T}(B(A,M)[1])=\bigoplus_{n\ge 1}B(A,M)[1]^{\otimes n},\] 
where $Q_1$ is the coderivation induced by $q_1=d\colon B(A,M)[1]\to B(A,M)[1]$ and $Q_2$ is the coderivation 
induced by the map 
\[q_2\colon B(A,M)[1]\otimes B(A,M)[1]\to B(A,M)[1],\qquad q_2(b,c)=(-1)^{\bar{b}}bc\,.\]
Notice that for $x\in A\oplus M$ the degree of $xp(t)dt$ in $B(A,M)[1]$ is equal to the degree of $x$ in $A\oplus M$.

We need to prove that $d\int_{\infty}=\int_{\infty}Q$. Notice that 
\[Q_1(R^{\otimes i_1}\otimes S^{\otimes j_1}\otimes \cdots \otimes
R^{\otimes i_h}\otimes S^{\otimes j_h})\subset 
R^{\otimes i_1}\otimes S^{\otimes j_1}\otimes \cdots\otimes 
R^{\otimes i_h}\otimes S^{\otimes j_h},\]
for every choice of indices and 
$Q_2(R^{\otimes i_1}\otimes S^{\otimes j_1}\otimes \cdots \otimes
R^{\otimes i_h}\otimes S^{\otimes j_h})$ is contained in the direct sum of $h$ 
similar components, each one with the indices $i_s$ lowered by $1$. Moreover, since $q_2(S\otimes R)=0$, 
the only non trivial cases to consider are:
\[ Q_1\colon R^{\otimes n}\otimes S\to R^{\otimes n}\otimes S,\qquad
Q_2\colon R^{\otimes n}\otimes S\to R^{\otimes n-1}\otimes S.\]
In order to prove that, for an element 
\[ x\in R_{i_1}\otimes\cdots\otimes R_{i_n}\otimes S_{i_{n+1}},\qquad i_j=0,1\;,\]
we have $d\int_{\infty}x=\int_{\infty}Q(x)$, we first point out that, if there are at least two indices $i_j$ equal to $0$, then $\int_{\infty}x=\int_{\infty}Q(x)=0$. Therefore we have to consider only the following four cases: 
\begin{enumerate}

\item $x\in R_1^{\otimes n}\otimes S_1$,

\item $x\in R_1^{\otimes n}\otimes S_0$,

\item $x\in R_0\otimes R_1^{\otimes n-1}\otimes S_0$,

\item $x\in R_1^{\otimes h}\otimes R_0\otimes R_1^{\otimes n-h-1}\otimes  S_1$, for $h=1,\ldots, n-1$.
\end{enumerate}

In the first case $Q(x)=Q_1(x)$ and the equality $d\int_{\infty}x=\int_{\infty}Q_1(x)$ follows immediately by Leibniz rule. 
In the second case, by linearity we may assume 
\[ x=a_1q_1(t)dt\otimes\cdots\otimes a_nq_n(t)dt\otimes mp(t),\qquad p(0)=p(1)=0\;.\]
In this case we have $\int_{\infty}x=0$,  
\[ \begin{split}\int_{\infty}Q_1(x)&=(-1)^{\overline{a_1\cdots a_nm}-1}
\int_{\infty}a_1q_1(t)dt\otimes\cdots\otimes a_nq_n(t)dt\otimes mp'(t)dt\\
&=(-1)^{\overline{a_1\cdots a_nm}-1}a_1\cdots a_nm\,\int_{n+1}(q_1,\ldots,q_n,p')\end{split}
\]
\[ \begin{split}\int_{\infty}Q_2(x)&=(-1)^{\overline{a_1\cdots a_n}}
\int_{\infty}a_1q_1(t)dt\otimes\cdots\otimes a_nq_n(t)dt\, mp(t)\\
&=(-1)^{\overline{a_1\cdots a_nm}-1}
\int_{\infty}a_1q_1(t)dt\otimes\cdots\otimes a_nmq_n(t)p(t)dt\\
&=(-1)^{\overline{a_1\cdots a_nm}-1}a_1\cdots a_nm\,\int_{n}(q_1,\ldots,q_np)\end{split}
\]
and therefore $\int_{\infty}Q(x)=0$. In the third case, by linearity we may assume 
\[ x=a_1q_1(t)\otimes\cdots\otimes a_nq_n(t)dt\otimes mp(t)dt,\qquad q_1(0)=0\;.\]
As above $\int_{\infty}x=0$, 
\[ \begin{split}\int_{\infty}Q_1(x)&=(-1)^{\overline{a_1}-1}
\int_{\infty}a_1q_1'(t)dt\otimes\cdots\otimes a_nq_n(t)dt\otimes mp(t)dt\\
&=(-1)^{\overline{a_1}-1}a_1\cdots a_nm\,\int_{n+1}(q_1',\ldots,q_n,p),\end{split}
\]
\[ \begin{split}\int_{\infty}Q_2(x)&=(-1)^{\overline{a_1}}
\int_{\infty}a_1q_1(t)a_2q_2(t)dt\otimes\cdots\otimes a_nq_n(t)dt\otimes mp(t)dt\\
&=(-1)^{\overline{a_1}}
\int_{\infty}a_1a_2q_1q_2(t)dt\otimes\cdots\otimes a_nq_n(t)dt\otimes mp(t)dt\\
&=(-1)^{\overline{a_1}}a_1\cdots a_nm\,\int_{n}(q_1q_2,\ldots,q_n,p),\end{split}
\]
and then $\int_{\infty}Q(x)=0$. The proof of the last case is completely similar and left to the reader.
The linear part of $\int_{\infty}$ is a quasi-isomorphism of complexes by Lemma~\ref{lemma.tw-quotient}.
\end{proof}

It is obvious that the $A_\infty$ morphism of Theorem~\ref{thm.iterated} is functorial in the following sense: 
given a morphism of differential graded associative algebras $A_1\to A_2$, 
a left $A_1$-module $M_1$, a left $A_2$-module $M_2$ and a morphism of $A_1$-modules 
$M_1\to M_2$ we get a commutative diagram of $A_{\infty}$ algebras
\[ \xymatrix{B(A_1,M_1)\ar[r]\ar[d]^{\int_\infty}&B(A_2,M_2)\ar[d]^{\int_\infty}\\
M_1[-1]\ar[r]&M_2[-1].}\]

Assume for simplicity that $A_1\subseteq A_2$ and $M_1\subseteq M_2$; according to Definition~\ref{def.twdoublefiber},  the double Thom-Whitney homotopy fiber of the commutative square 
\[ \xymatrix{A_1\ar[r]\ar[d]&A_1\oplus M_1\ar[d]\\
A_2\ar[r]&A_2\oplus M_2}\]
is  the differential graded associative subalgebra 
\[ TW^{[2]}(A_2/A_1,M_2/M_1)\subseteq  (A_2\oplus M_2)[s,t,ds,dt]\]
of pairs $(a(s,t),m(s,t))$ such that 
\[ a(0,t)=a(s,0)=m(0,t)=m(s,0)=0,\]
\[m(s,1)=0,\quad a(1,t)\in A_1[t,dt],\quad m(1,t)\in M_1[t,dt]\,.\] 
Equivalently, $TW^{[2]}(A_2/A_1,M_2/M_1)=B(C,N)$, where
\[ C=\{a(s)\in A_2[s,ds]\mid a(0)=0,\; a(1)\in A_1\},\quad N=\{m(s)\in M_2[s,ds]\mid m(0)=0,\; m(1)\in M_1\}\,.\]
Therefore the morphism $B(C,N)\to N[-1]$ of Theorem~\ref{thm.iterated} gives an $A_\infty$ quasi-isomorphism.
\[ TW^{[2]}(A_2/A_1,M_2/M_1)\xrightarrow{\;\int_{\infty}\;}\{ m(s)\in M_2[s,ds][-1]\mid m(0)=0,\, m(1)\in M_1\}.\]
The composition of $\int_{\infty}$ with the usual integration 
\[ \int_0^1\colon M_2[s,ds][-1]\to \frac{M_2}{M_1}[-2]\]
gives an $A_{\infty}$ quasi-isomorphism 
\[TW^{[2]}(A_2/A_1,M_2/M_1)\to \frac{M_2}{M_1}[-2]\,.\]

When $(V,F)$ is  a DG-pair, $A_1=\End^*(V;F)$, $A_2=\End^*(V)$, $M_1=F$ and $M_2=V$, the above computation gives immediately the following theorem.

\begin{theorem}\label{thm.iterartiperjacobi} 
Let $(V,F)$ be a DG-pair and 
consider the complex $V[s,ds][-1]$ as a differential associative algebra with trivial product. Then we have 
an $A_{\infty}$ quasi-isomorphism
\[TW^{[2]}(\End^*(V)/\End^*(V;F),V/F)\to \frac{V}{F}[-2]\,,\]
defined explicitly by setting
\[ (\End^*(V)[s,ds,t]dt)^{\otimes n}\otimes V[s,ds,t]dt\to V/F[-1],\qquad n\ge 0,\]
\begin{multline*}
 a_1p_1(s,t)dt\otimes \cdots\otimes a_np_n(s,t)dt\otimes  vp(s,t)dt\mapsto\\
\mapsto a_1a_2\cdots a_nv\int_{0\le s\le 1}\int_{0\le t_1\le \cdots\le t_n\le 1}
p_1(s,t_1)\cdots p_n(s,t_n)dt_1\cdots dt_n\,,\end{multline*} 
\[ a_i\in \End(V),\qquad v\in V,\qquad p_i(s,t)\in\K[s,ds,t],\]
and extending by $0$ to the other direct summands of 
$\overline{T}(TW^{[2]}(\End^*(V)/\End^*(V;F),V/F))$.
\end{theorem}

Since the Lie bracket on $\Aff(V)=\End^*(V)\oplus V$ is the graded commutator of the associative product 
$(f,v)(g,u)=(fg,f(u))$, the commutator bracket in the associative algebra 
$TW^{[2]}(\End^*(V)/\End^*(V;F),V/F)$ gives exactly the DG-Lie Thom-Whitney double homotopy fiber of the 
Jacobian diagram of the pair $(V,F)$. By well known results about symmetrization of $A_{\infty}$-algebras \cite{LadaMarkl}, the composite of the symmetrization map 
\[ \overline{S}(TW^{[2]}(\End^*(V)/\End^*(V;F),V/F))\to \overline{T}(TW^{[2]}(\End^*(V)/\End^*(V;F),V/F)), \]  
\[ x_1\odot\cdots\odot x_n\mapsto \sum_{\sigma\in \Sigma_n}\epsilon(\sigma)x_{\sigma(1)}\otimes\cdots\otimes x_{\sigma(n)},\]
with the morphism  
\[ \overline{T}(TW^{[2]}(\End^*(V)/\End^*(V;F),V/F))\to \frac{V}{F}[-1], \]  
gives an $L_{\infty}$ quasi-isomorphism $TW^{[2]}(\End^*(V)/\End^*(V;F),V/F)\to V/F[-2]$.

In conclusion, given a formal Abel-Jacobi datum $(\mathfrak{g},\tilde{\mathfrak{g}},V,F,v,\bi)$ 
the above theorem allows, at least in principle, to write explicitly the formal Abel-Jacobi map  on Maurer-Cartan elements  
\[
\mathcal{AJ}_{\MC}\colon \MC_{\tilde{\mathfrak{g}}\hookrightarrow \mathfrak{g}}\to \MC_{V/F[-2]},\]
as the composition of the following maps (for simplicity of notation we omit to write the maximal ideals of the Artin rings):
\begin{enumerate}

\item  the map of Lemma~\ref{lem.modellopiccolofibraomotopica}:
\[ \MC_{\tilde{\mathfrak{g}}\hookrightarrow \mathfrak{g}}=\{x\in \mathfrak{g}^0\mid e^x\ast 0\in \tilde{\mathfrak{g}}\}\to 
\MC_{TW(\tilde{\mathfrak{g}}\hookrightarrow \mathfrak{g})}=\{x(t)\in \MC_{\mathfrak{g}[s,ds]}\mid x(0)=0,\; x(1)\in \tilde{\mathfrak{g}}\},\]
\[ x\mapsto e^{sx}\ast 0\,.\]

\item the map of Lemmas~\ref{isomorphisms.def} and  \ref{lem.sezionecartan}:
\[ \MC_{\mathfrak{g}[s,ds]}\to \MC_{\cone(\Id_{\mathfrak{g}[s,ds]})[t,dt]},\qquad 
x(s)\mapsto e^{-t\flat x(s)}\ast 0\,.\]

\item the Cartan homotopy $\bi^v\colon \mathfrak{g}[s,ds,t,dt]\to \Aff(V)[s,ds,t,dt]$;

\item the $L_{\infty}$ morphism of Theorem~\ref{thm.iterartiperjacobi}.

\end{enumerate}

\section*{Acknowledgments} 
We thank Donatella Iacono and Jim Stasheff for fruitful conversation about the previous versions of this paper, and the referees for their useful comments. Both authors 
acknowledge the very partial support by Italian MIUR under PRIN project 2015ZWST2C ``Moduli spaces and Lie theory''.

\end{document}